\documentclass[reqno,11pt,onesided]{article}
 \usepackage{amsmath,amsthm,amssymb}
\usepackage{verbatim}
\usepackage{bibspacing}
\usepackage{hyperref}

\date{}

 \newcommand{\norm}[1]{\left\Vert#1\right\Vert}
\newcommand{\snorm}[1]{\Vert#1\Vert}
\newcommand{\abs}[1]{\left\vert#1\right\vert}
\newcommand{\set}[1]{\left\{#1\right\}}
\newcommand{\brac}[1]{\left(#1\right)}
\newcommand{\scalar}[1]{\left \langle #1 \right \rangle}

\newcommand{\Real}{\mathbb{R}}
\newcommand{\E}{\mathbb{E}}
\newcommand{\med}{\text{med}}
\renewcommand{\P}{\mathbb{P}}
\newcommand{\Poin}{\text{\rm{Poin}}}
\newcommand{\Lin}{\text{\rm{Lin}}}
\newcommand{\Che}{\text{\rm{Che}}}
\newcommand{\eps}{\epsilon}

\newcommand{\K}{\mathcal{K}}

\newcommand{\Vol}{\mbox{\rm{Vol}}}

\renewcommand{\H}{\mathcal{H}}

\newcommand{\Var}{\mathbb{V}\text{\rm{ar}}}

\DeclareMathOperator*{\argmax}{arg\,max}
\DeclareMathOperator*{\esssup}{ess\,sup}

\def\XXint#1#2#3{{\setbox0=\hbox{$#1{#2#3}{\int}$}
\vcenter{\hbox{$#2#3$}}\kern-.5\wd0}}

\newtheorem{thm}{Theorem}[section]
\newtheorem*{thm*}{Theorem}
\newtheorem{cor}[thm]{Corollary}
\newtheorem{lem}[thm]{Lemma}
\newtheorem{prop}[thm]{Proposition}
\newtheorem*{prop*}{Proposition}

\newtheorem*{problem*}{Problem}

\newtheorem*{conj*}{Conjecture}

\newtheorem*{dfn*}{Definition}
\theoremstyle{definition}
\newtheorem{rem}[thm]{Remark}
\newtheorem*{rem*}{Remark}
\newtheorem*{fact*}{Fact}
\newtheorem{example}[thm]{Example}
\theoremstyle{proof}

\renewcommand{\refname}{References}

\numberwithin{equation}{section}

\begin{document}

\title{The KLS Isoperimetric Conjecture for\\ Generalized Orlicz Balls}
\author{Alexander V. Kolesnikov\textsuperscript{1} and Emanuel Milman\textsuperscript{2}}
\footnotetext[1]{Faculty of Mathematics, Higher School of Economics, Moscow, Russia. 
The author was supported by RFBR project 17-01-00662 and DFG project RO 1195/12-1. The article was prepared 
within the framework of the Academic Fund Program at the National Research University Higher School of
Economics (HSE) in 2017-2018 (grant No 17-01-0102) and by the Russian Academic Excellence Project ``5-100". 
Emails:  akolesnikov@hse.ru, sascha77@mail.ru.}
\footnotetext[2]{Department of Mathematics, Technion - Israel
Institute of Technology, Haifa 32000, Israel. The research leading to these results is part of a project that has received funding from the European Research Council (ERC) under the European Union's Horizon 2020 research and innovation programme (grant agreement No 637851). Email: emilman@tx.technion.ac.il.}

\begingroup    \renewcommand{\thefootnote}{}    \footnotetext{2010 Mathematics Subject Classification: 60D05, 52A23, 46B07.}
    \footnotetext{Keywords: KLS Conjecture, Spectral-Gap, Convex Bodies, Generalized Orlicz Balls.}
\endgroup

\maketitle

\begin{abstract}
.~What is the optimal way to cut a convex bounded domain $K$ in Euclidean space $(\Real^n,\abs{\cdot})$ into two halves of equal volume, so that the interface between the two halves has least surface area? A conjecture of Kannan, Lov\'asz and Simonovits asserts that, if one does not mind gaining a universal numerical factor (independent of $n$) in the surface area, one might as well dissect $K$ using a hyperplane. This conjectured essential equivalence between the former non-linear isoperimetric inequality and its latter linear relaxation, has been shown over the last two decades to be of fundamental importance to the understanding of volume-concentration and spectral properties of convex domains. 
In this work, we address the conjecture for the subclass of generalized Orlicz balls
\[
K = \{x \in \Real^n \; ; \; \sum_{i=1}^n V_i(x_i) \leq E \} ,
\] 
confirming its validity for certain levels $E \in \Real$ under a mild technical assumption on the growth of the convex functions $V_i$ at infinity (without which we confirm the conjecture up to a $\log(1+n)$ factor). In sharp contrast to previous approaches for tackling the KLS conjecture, we emphasize that no symmetry is required from $K$. This significantly enlarges the subclass of convex bodies for which the conjecture is confirmed. 
\end{abstract}

\section{Introduction}

\subsection{A Conjecture of Kannan--Lov\'asz--Simonovits}

Given a separable metric space $(X,d)$ endowed with a Borel probability measure $\mu$, Minkowski's
(exterior) boundary measure of a Borel set $A \subset X$, denoted $\mu^+(A)$, is
defined as $\mu^+(A) := \liminf_{\eps \to 0} \frac{\mu(A^d_{\eps}) -\mu(A)}{\eps}$, where $A^d_{\eps} := \set{x \in X ; \exists y \in A \;\; d(x,y) < \eps}$ denotes the $\eps$-neighborhood of $A$ in $(X,d)$. The Cheeger constant is then defined as:
\begin{equation} \label{eq:Che}
D_{\Che}(X,d,\mu) := \inf_{A \subset X} \frac{\mu^+(A)}{\min(\mu(A),1-\mu(A))} ~,
\end{equation}
measuring a certain isoperimetric property of the space $(X,d,\mu)$. In this work, we restrict our scope to the Euclidean setting $(X,d) = (\Real^n,\abs{\cdot})$, and simply write $D_{\Che}(\mu) = D_{\Che}(\Real^n,\abs{\cdot},\mu)$. In the latter linear setting, we can also introduce the following \emph{linear relaxation} of the Cheeger constant, defined as:
\[
D_{\Che}^{\Lin}(\mu) := \inf_{\footnotesize \begin{array}{c} H \subset \Real^n\\H \text{ is a halfspace} \end{array}} \frac{\mu^+(H)}{\min(\mu(H),1-\mu(H))} .
\]
Note that when $A$ has smooth boundary and $\mu$ is supported on a set $\Omega$ having Lipschitz boundary and has continuous density $\Psi$ in $\Omega$, then $\mu^+(A) = \int_{\partial A \cap \text{int}(\Omega)} \Psi(x) d\H^{n-1}(x)$, where $\H^k$ denotes the $k$-dimensional Hausdorff measure.

Clearly $D_{\Che}^{\Lin}(\mu) \geq D_{\Che}(\mu)$, and in general it is not hard to see that this inequality cannot be reversed in any weak sense, as the right-hand-side may be zero.  
However, when $\mu = \lambda_K$, where $\lambda_K$ denotes the uniform (Lebesgue) probability measure on $K \subset \Real^n$,  a convex compact set with non-empty interior (``convex body"), Kannan, Lov\'asz and Simonovits (KLS) conjectured in \cite{KLS} (using an equivalent formulation) that:
\begin{equation} \label{eq:KLS}
D_{\Che}(\lambda_K)  \geq c D_{\Che}^{\Lin}(\lambda_K) ,
\end{equation}
for some universal numeric constant $c > 0$, independent of any other parameters such as $n$ or $K$. We reserve in this work the use of $c,C,C_1,C_2,c',C',C''$ etc...~to denote such positive universal numeric constants. 

\smallskip

Recall that a measure $\mu$ on $\Real^n$ is called log-concave if $\mu = \exp(-V(x)) dx$ with $V : \Real^n \rightarrow \Real \cup \set{+\infty}$ convex; in particular, $\mu = \lambda_K$ is log-concave. The class of log-concave probability measures on affine subspaces of $\Real^N$ for all $N \geq 1$ is the smallest class containing $\lambda_K$ for all convex bodies $K \subset \Real^N$ (for all $N \geq 1$) which is in addition closed under taking marginals and weak limits (see e.g. \cite{SpanishBook}). 
It is not hard to see that a positive answer to the KLS conjecture would also lead to a positive answer to the analogous question for the entire class of log-concave probability measures, so   it is also interesting to study the conjecture in this extended generality.

\begin{rem*}
 It is known that for a log-concave measure $\mu$, the infimum in (\ref{eq:Che}) is attained for a Borel set $A$ of measure $1/2$ (see Sternberg--Zumbrun \cite{SternbergZumbrun} for $\lambda_K$, Bobkov \cite{BobkovExtremalHalfSpaces} for the one-dimensional case and \cite{EMilman-RoleOfConvexity} in general), and the same applies to its linear relaxation:
\begin{align*}
D_{\Che}(\mu) & = 2 \inf_{A \subset \Real^n} \set{ \mu^+(A) \; ; \; \mu(A) = 1/2 } ~,\\
D_{\Che}^{\Lin}(\mu) & = 2 \inf_{\footnotesize \begin{array}{c} H \subset \Real^n\\H \text{ is a halfspace} \end{array}} \set{ \mu^+(H) \; ; \; \mu(H) = 1/2 } . 
\end{align*}
So the KLS conjecture ultimately pertains to the isoperimetric behaviour of sets having measure $1/2$.
\end{rem*}

The KLS conjectured essential equivalence between the former non-linear isoperimetric inequality and its latter linear relaxation, has been shown over the last two decades to be of fundamental importance to the understanding of volumetric and spectral properties of convex domains, revealing numerous connections to other central conjectures on the concentration of volume in convex bodies (see e.g. \cite{LedouxSpectralGapAndGeometry,BobkovKoldobsky,FleuryOnVarianceConjecture,LatalaJacobInfConvolution,BallNguyen-KLSImpliesSlicing,EldanKlartagThinShellImpliesSlicing,Eldan-StochasticLocalization} or the monograph \cite{SpanishBook} for a nice overview). Let us only mention here the following equivalent formulation of the KLS conjecture, which has a clear analytic interpretation. 

Denote by $D_{\Poin}(\mu)$ the Poincar\'e constant of $\mu$, namely the best possible constant in the following Poincar\'e inequality:
\begin{equation} \label{eq:Poincare}
\snorm{f - \int f d\mu}_{L^2(\mu)} \leq D_{\Poin}(\mu) \norm{\abs{\nabla f}}_{L^2(\mu)} \;\;\; \text{for all Lipschitz $f : \Real^n \rightarrow \Real$} .
\end{equation}
When $\mu = \lambda_K$,  $D_{\Poin}(\lambda_K) = 1 / \sqrt{\lambda_1(K)}$ where $\lambda_1(K)$ denotes the first non-zero eigenvalue of the Neumann Laplacian on $K$ (a similar interpretation holds for a general $\mu$ using an appropriate weighted Laplacian). We denote by $D_{\Poin}^{\Lin}(\mu)$ the linear relaxation obtained by only testing (\ref{eq:Poincare}) on linear functionals $f(x) = \scalar{x,\theta}$; clearly $D_{\Poin}^{\Lin}(\mu) \leq D_{\Poin}(\mu)$. It is known by results of Maz'ya \cite{MazyaSobolevImbedding}, Cheeger \cite{CheegerInq}, Buser \cite{BuserReverseCheeger} and Ledoux \cite{LedouxSpectralGapAndGeometry}, that for all log-concave probability measures $\mu$ on $\Real^n$:
\[
 \frac{1}{2} D_{\Che}(\mu) \leq \frac{1}{D_{\Poin}(\mu)} \leq C \; D_{\Che}(\mu),
\]
for some universal constant $C > 1/2$; the same inequality also holds for the corresponding linear relaxations $D_{\Che}^{\Lin}(\mu)$ and $D_{\Poin}^{\Lin}(\mu)$. Consequently, the KLS conjecture may be equivalently reformulated as asserting that:
\begin{equation} \label{eq:KLS2}
D_{\Poin}(\mu) \leq C \; D_{\Poin}^{\Lin}(\mu) ,
\end{equation}
for some universal constant $C > 1$ and all log-concave measures $\mu$. In other words, the KLS conjecture asserts that for log-concave measures (and in particular, on convex bodies), the Poincar\'e inequality (\ref{eq:Poincare}) should be \emph{essentially} saturated by linear functionals.

\subsection{Previously Known Results}

More than two decades after being put forth, the KLS conjecture is still unresolved, and the presently best known (dimension-dependent) estimate on $C = C_n$ in (\ref{eq:KLS2}) 
is $C_n \leq C n^{1/4}$, obtained very recently (after this work was posted on the arXiv) by Y.~T.~Lee and S.~Vempala \cite{LeeVempala-KLS} by employing the remarkable Stochastic Localization method of R.~Eldan \cite{Eldan-StochasticLocalization}; previous contributions include those by KLS \cite{KLS}, S.~Bobkov \cite{BobkovVarianceBound}, B.~Klartag \cite{KlartagCLP,KlartagCLPpolynomial}, B.~Fleury \cite{FleuryImprovedThinShell} and O.~Gu\'edon and Milman \cite{GuedonEMilmanInterpolating}. 
The conjecture has been confirmed (uniformly in $n$) for unit-balls of $\ell_p^n$ (by S.~Sodin \cite{SodinLpIsoperimetry} when $p \in [1,2]$ and R.~Lata{\l}a and J.~Wojtaszczyk \cite{LatalaJacobInfConvolution} when $p \in [2,\infty]$), the simplex by F.~Barthe and P.~Wolff \cite{BartheWolffGammaDistributions}, convex bodies of revolution by N.~Huet \cite{HuetSphericallySymmetric}, convex sets of bounded volume-ratio constructed in a certain manner from log-concave measures which satisfy the conjecture \cite{KolesnikovEMilman-HardyKLS}, linear images and Cartesian products of these subclasses (see Bobkov--Houdr\'e \cite{BobkovHoudre} for the latter) and various perturbations thereof \cite{EMilman-RoleOfConvexity,EMilmanGeometricApproachPartII}. 
For the interesting class of unconditional convex bodies (invariant under reflections with respect to the coordinate hyperplanes), the best known estimate $C_n = C \log(1+n)$ was established by B.~Klartag \cite{KlartagUnconditionalVariance}. In addition, the conjecture has been established in a certain weak sense for random Gaussian polytopes (with high-probability) by B.~Fleury \cite{Fleury-ThinShellForRandomBodies}. 

\smallskip

Besides these subclasses of convex bodies and their natural extensions to the log-concave setting, the extended KLS conjecture has also been confirmed for rotation invariant log-concave measures by S.~Bobkov \cite{Bobkov-SpectralGapForSphericallySymmetric} (see also \cite{HuetSphericallySymmetric} for generalizations), for log-concave measures with strictly convex potentials $V$ by Bakry--\'Emery \cite{BakryEmery}, for certain Gibbs measures corresponding to conservative spin systems by Barthe--Wolff \cite{BartheWolffGammaDistributions} and Barthe--Milman \cite{BartheEMilmanConservativeSpins}, for certain log-concave measures supported in a cube by Klartag \cite{Klartag-ConcentrationOfMeasureOnCube}, and for unconditional measures with strictly positive derivatives in the principle directions by the authors in \cite{KolesnikovEMilman-RiemannianMetrics}. In addition, Klartag's $C_n = C \log(1+n)$ estimate for unconditional log-concave measures has been generalized to log-concave measures enjoying more general symmetries by Barthe and D.~Cordero--Erausquin \cite{BartheCorderoVariance}. To the best of our knowledge, this is essentially a complete list. 

\subsection{Generalized Orlicz Balls}

The above results typically make heavy use of the symmetries possessed by $K$ or $\mu$. In this work, we address the KLS conjecture for a certain family of convex bodies which may be called \emph{generalized Orlicz balls}. Contrary to the standard definition of these bodies in the literature (see e.g. \cite{SpanishBook}), we emphasize that our definition does not impose any symmetry conditions on these bodies. 

\begin{dfn*}
A convex body $K \subset \Real^n$ is called a generalized Orlicz ball if there exist $n$ one-dimensional convex functions $V_i : \Real \rightarrow \Real$ and $E \in \Real$ so that:
\[
K = \set{ x \in \Real^n \; ; \; \sum_{i=1}^n V_i(x_i) \leq E } .
\]
\end{dfn*}

The traditional definition also requires that $V_i$ be even functions which vanish at the origin, so that the resulting class is always unconditional - we will call such bodies \emph{unconditional generalized Orlicz balls}. In that case, $K$ is the unit-ball of the generalized Orlicz norm:
\begin{equation} \label{eq:norm}
\norm{x}_{K} := \inf \set{ t > 0 \; ; \; \sum_{i=1}^n V_i(x_i/t) \leq E } ;
\end{equation}
indeed, the convexity of $V_i$ ensures the validity of the triangle inequality, and the symmetry of $V_i$ ensures that $\norm{-x}_K = \norm{x}_K$, so that this defines a norm (with an unconditional basis). By abuse of notation, we will still refer to (\ref{eq:norm}) as a norm as soon as $K$ contains the origin in its interior, even without any symmetry assumptions on $V_i$. As shown by Wojtaszczyk \cite{CoordCorrelationForOrliczNorms}, contrary to general unconditional convex bodies, unconditional generalized Orlicz balls enjoy the following negative correlation property (first noted by Anttila--Ball--Perissinaki \cite{ABP} for unit-balls of $\ell_p^n$):
\begin{equation} \label{eq:corr}
\E X_i^2 X_j^2 \leq \E X_i^2 \E X_j^2 \;\;\; \forall i \neq j ,
\end{equation}
where $X$ is a random-vector uniformly distributed in $K$. 
Naturally, this property heavily relies on the underlying symmetry, and is very helpful in establishing various concentration properties for this class; for instance, using an extension of (\ref{eq:corr}) due to Pilipczuk--Wojtaszczyk \cite{PilipczukWojtaszczyk}, Fleury \cite{FleuryOnVarianceConjecture} showed that for unconditional generalized Orlicz balls, $\abs{X}$ is optimally concentrated around its mean. However, to the best of our knowledge, even for this subclass of unconditional bodies, the best estimate on $C = C_n$ in the KLS conjecture (\ref{eq:KLS2}) is the general one for unconditional bodies $C_n = C \log(1+n)$ due to Klartag \cite{KlartagUnconditionalVariance}. 

\subsection{Simplified Main Results}

In this work, we \emph{do not} impose any symmetry assumptions on $V_i$, and in particular \emph{do not} (and cannot) employ (\ref{eq:corr}) at all. We formulate our main results in full generality in the next section, but for now we only state the following simplified version: 
\begin{thm}[Simplified Main Theorem] \label{thm:intro}
For each $i=1,\ldots,n$, let $V_i : \Real \rightarrow \Real$, $i=1,\ldots,n$, denote a convex function normalized so that $\min V_i = 0$ and so that $\mu_i := \exp(-V_i(y)) dy$ is a probability measure on $\Real$ with barycenter at the origin. Given $E > 0$, set:
\begin{equation} \label{eq:KE}
K_E := \set{ x \in \Real^n \; ; \; \sum_{i=1}^n V_i(x_i) \leq E} .
\end{equation}
Let $X_i$ denote random-variables distributed according to $\mu_i$, and set:
\begin{equation} \label{eq:EV-XI}
E_V := 1 + \sum_{i=1}^n \E V_i(X_i) . 
\end{equation}
Then $E_V \leq n+1$, and for $E = E_V$ we have:
\begin{equation} \label{eq:intro-vol}
\frac{1}{C} \leq \Vol(K_E)^{\frac{1}{n}} \leq C ,
\end{equation}
and:
\begin{equation} \label{eq:intro-result}
 D_{\Poin}(\lambda_{K_E}) \leq C \log(e+ A^{(2)} \wedge n) D_{\Poin}^{\Lin}(\lambda_{K_E}) .
\end{equation}
Here $C>1$ is a universal constant, and:
\begin{equation} \label{eq:intro-psi1}
A^{(2)} := \frac{1}{\sqrt{n}} \norm{(\alpha^{(2)}_i)_{i=1}^n}_2 ~,~\alpha^{(2)}_i := \norm{ V_i'(y) y }_{L^2(\mu_i)} . 
\end{equation}
\end{thm}

In particular, we confirm the KLS conjecture for the generalized Orlicz ball $K_E$ as soon as $A^{(2)}$ is bounded above by a constant, reflecting a certain upper bound on the rate of growth of $\set{V_i}$ at infinity. The volume estimate (\ref{eq:intro-vol}) is a natural expected normalization which serves as a sanity check, preventing various trivial statements (such as when $E \rightarrow 0$). The precise result we formulate in Section \ref{sec:results} provides a more flexible explicit description of the levels $E$ to which the above result applies -- see Remark \ref{rem:Levels} below. This provides an explicitly computable criterion for the validity of the KLS conjecture, and significantly extends the class of convex bodies for which the conjecture is confirmed. 

\smallskip

Note that the dependence on $A^{(2)}$ in (\ref{eq:intro-result}) is logarithmic, and that in the worst case, regardless of the value of $A^{(2)}$ (which may be infinite, see e.g. \cite[Example 1]{Krugova-BadConvexFunctions}), 
the estimate (\ref{eq:intro-result}) confirms the KLS conjecture for $K_E$ up to a factor of  $\log(1+n)$,
 matching Klartag's estimate for unconditional convex bodies, but without assuming any symmetry. In fact, the above $\log(1+n)$ factor is a consequence of a more general result, stating that one may always find a level set of a \emph{genera}l log-concave measure $\mu$ (no product structure assumed) having essentially the same spectral-gap, up to this factor:

\begin{thm}[From log-concave measure to good level-set] \label{thm:intro-logn}
Let $\mu = \exp(-V(x)) dx$ denote a log-concave probability measure on $\Real^n$ with $\min V = 0$. Let $X$ denote a random-vector distributed according to $\mu$, and set:
\[
E_V := 1 + \E V(X) . 
\]
Then $E_V \leq n+1$, and for $E = E_V$, $K_E = \set{x \in \Real^n \; ;\; V(x) \leq E}$ satisfies (\ref{eq:intro-vol}) and:
\[
D_{\Poin}(\lambda_{K_E}) \leq C D_{\Poin}(\mu) \log(e + \sqrt{n} D_{\Poin}(\mu)) .
\]
\end{thm}

\begin{rem} \label{rem:Levels}
The results of Theorems \ref{thm:intro} and \ref{thm:intro-logn} apply to all levels $E$ in the following explicit set:
\begin{equation} \label{eq:intro-intro-level-def}
\text{Level}(V) := \set{E \geq 0 \; ; \; e^{-E} \Vol(K_E) \geq \frac{1}{e} \frac{n^n e^{-n}}{n!}} .
\end{equation}
Proposition \ref{prop:Level} ensures that this set is a non-empty interval $[E_{\min},E_{\max}]$ with $E_{\min} \leq n$ and $1 \leq E_{\max} - E_{\min} \leq e \sqrt{2 \pi n} (1 + o(1))$ as $n \rightarrow \infty$, and that $E_V \in \text{Level}(V)$. The constant $\frac{1}{e}$ in front of the term $\frac{n^n e^{-n}}{n!}$ above may be replaced by $\frac{1}{e^q}$ for any fixed $q \geq 1$ (and for some statements in this work, $q \geq 0$), resulting only in different numeric constants in our results; this variant of (\ref{eq:intro-intro-level-def}) is denoted by $\text{Level}^{(q)}(V)$. If one employs $q = n / (E_V-1)$, it is not hard to show that the (perhaps more natural) level $E_V - 1$ lies in $\text{Level}^{(q)}(V)$ and thus our results apply to it as well -- but we do not pursue this nuance here. 
\end{rem}

More general versions of these results (dispensing with the restrictions that $\min V_i = 0$, that $\mu_i$ are probability measures, and that their barycenter is at the origin) will be presented in Section \ref{sec:results} and Subsection \ref{subsec:gen}. In this introductory section, we provide a couple of simple examples to illustrate how these (extended) results may be applied; their analysis is deferred to Subsection \ref{subsec:examples}.
We denote by $a_+ := (\abs{a}+a)/2$ and $a_- := (\abs{a} - a)/2$ the positive and negative parts of $a \in \Real$.  

\begin{example} \label{ex:1}
Let $p^{\pm}_i \in [1,P]$, $i=1,\ldots,n$, for some $P \geq 1$, and set:
\[
V_i(x_i) := (x_i)_+^{p^+_i} + (x_i)_-^{p^-_i} . 
\]
Let $E_V$ be defined by (\ref{eq:EV-XI}), where $X_i$ are distributed according to $\mu_i$ having density proportional to $\exp(-V_i)$. 
Then for $E = E_V$, the generalized Orlicz ball (\ref{eq:KE}) satisfies (\ref{eq:intro-vol}) and:
\begin{equation} \label{eq:ex1-1}
D_{\Poin}(\lambda_{K_E}) \leq C \log(e+ P) D_{\Poin}^{\Lin}(\lambda_{K_E}) ,
\end{equation}
for some universal $C > 0$. Moreover, if $p^{\pm}_i \in [2,P]$, then:
\[
D_{\Poin}(\lambda_{K_E}) \leq C' \sqrt{\log(e+ P)} D_{\Poin}^{\Lin}(\lambda_{K_E}) .
\]
In particular, for fixed $P \geq 1$, this confirms the KLS conjecture for the bodies $K_E$ uniformly in $n \geq 1$.
Of course, one may replace the function $y^p$ in this example with other non-homogeneous variations like $y^p \log(1+y)$, etc... 
More generally, as suggested to us by the referee, it is worth pointing out that (\ref{eq:ex1-1}) remains valid (with $C$ depending solely on $c_1,c_2$ below) when the convex functions $V_i$ satisfy $\min V_i = V_i(0) = 0$,
\[
\forall i=1,\ldots,n  \;\;\; 0 < c_1 \leq \int_0^\infty \exp(-V_i(\pm x_i)) dx_i \leq c_2  < \infty , 
\]
and the following ``generalized doubling condition" holds:
\[ 
\forall i=1,\ldots,n \;\; \; \exists \eps_i > 0 \;\; \; \forall x_i \in \Real \;\;\;  V_i((1+\eps_i)x_i) \leq (1+\eps_i P) V_i(x_i)  .
\]
\end{example}

\begin{example} \label{ex:2}
Let $\mu = \exp(-V(x)) dx$ denote a log-concave probability measure on $\Real^n$ with $\min V = 0$ and $\text{Hess} V \geq \rho^2 \text{Id}$ with $\rho > 0$.
Then for $E = E_V \leq n+1$, $K_E = \set{x \in \Real^n \; ;\; V(x) \leq E}$ satisfies (\ref{eq:intro-vol}) and:
\[
D_{\Poin}(\lambda_{K_E}) \leq \frac{C}{\rho} \log(e + \sqrt{n} / \rho) .
\]
\end{example}

\subsection{Method of Proof}

Our approach is based on transferring concentration information from the log-concave measure $\mu := \exp(-V(x)) dx$ onto the uniform measure $\lambda_{K_E}$ on its level set $K_E$. We achieve this in three separate steps. The most important step is to transfer concentration from $\mu$ to an auxiliary measure $\mu_{K_E,w}$, which is a linearized version of $\mu$ supported on an annulus around $K_E$ of relative width $w = \frac{w_0}{\sqrt{n}}$. We then pass from $\mu_{K_E,w}$ to the cone measure $\sigma_{\partial K_E}$ supported on $\partial K_E$, from which we finally pass to $\lambda_{K_E}$ and optimize on $w_0 > 0$. Finally, an isoperimetric (or Poincar\'e) inequality is deduced using the convexity of $K_E$ and the known equivalence between concentration and isoperimetry under convexity assumptions. Surprisingly, these three different steps require three different methods for transferring concentration: an $L^p(\mu)$ estimate on $d\mu_{K_E,w}/d\mu$, a Wasserstein-distance estimate on $W_1(\mu_{K_E,w},\sigma_{\partial K_E})$, and a Hardy-type inequality with-boundary for $\lambda_{K_E}$. The only place where we need to assume that $\mu$ is a product measure (i.e. that $V(x) = \sum_{i=1}^n V_i(x_i)$ and hence $K_E$ is a generalized Orlicz ball) is in the first step, resulting in an estimate depending only on $A^{(2)}$ and not the dimension; this allows for future possible generalizations. To get the dimension-dependent $\log(1+n)$ estimate, no assumption on $\mu$ is needed beyond log-concavity, and we can simply use an $L^\infty$ estimate in the first step. 

\smallskip

The rest of this work is organized as follows. In Section \ref{sec:results} we formulate our various general main results in this work, of which Theorem \ref{thm:intro} is a particular case. In Section \ref{sec:Lp} we obtain the $L^p(\mu)$ estimate on $d\mu_{K_E,w}/d\mu$, modulo an estimate on $\Vol(K_E)$ which is obtained in Section \ref{sec:ZE}. In Section \ref{sec:W1} we obtain the Wasserstein distance estimate and the Hardy-type inequality. In Section \ref{sec:proofs} we put everything together and prove our main results. 

\smallskip

Further results pertaining to the distribution of $\Vol(K_{E})$ where $E = \sum_{i=1}^n V_i(X_i)$ and $X_i$ are independent random-variables distributed according to $\mu_i = \exp(-V_i) dx$, will be studied in a follow-up work by Barthe and Wolff \cite{BartheWolffOrliczBalls}. 

\smallskip
\textbf{Acknowledgement.} We thank Franck Barthe for informing us, after learning about a prior version of our results, that by employing a local Central-Limit Theorem, one can show that under mild assumptions on the functions $V_i$ and for large enough $n$ (depending on properties of $\{V_i\}$), the natural level $E_V - 1$ lies in $\text{Level}(V)$; this led us to notice that our proof actually shows that  $E_V \in \text{Level}(V)$ for any log-concave probability measure in $\Real^n$ and for any $n$. We also thank the anonymous referee for carefully and thoroughly reading the manuscript, for the constructive remarks which helped improve the presentation of the results, and for suggesting the more general setting in Example \ref{ex:1}.

\section{Statement of Results} \label{sec:results}

\begin{thm}[Main Technical Theorem] \label{thm:main}
Let $V_i : \Real \rightarrow \Real$, $i=1,\ldots,n$, denote a sequence of convex functions normalized so that $\mu_i := \exp(-V_i(y)) dy$ is a probability measure on $\Real$. 
Denote $V(x) = \sum_{i=1}^n V_i(x_i)$ and $m_i := \min V_i$, and assume that $\sum_{i=1}^n m_i = 0$ so that $\min V = 0$. 
Assume in addition that the following scale-invariant quantities are finite:
\begin{equation} \label{eq:main-psi1}
\forall i=1,\ldots,n \;\;\;  \alpha^{(\infty,2)}_i := (1+\norm{(V_i'(y) y)_-}_{L^\infty(\mu_i)}) \vee \norm{V_i'(y) y}_{L^2(\mu_i)} < \infty . 
\end{equation}
Set $A^{(\infty,2)} := \frac{1}{\sqrt{n}} \snorm{(\alpha^{(\infty,2)}_i)_{i=1}^n}_2$ and $M = \max_{i=1,\ldots,n} e^{m_i}$. 
Given $E > E_0 := V(0)$, define the following convex body on $\Real^n$ (containing the origin in its interior):
\[
K_E := \set{V \leq E} = \set{ x \in \Real^n \; ; \; \sum_{i=1}^n V_i(x_i) \leq E} ,
\]
the unit-ball of the generalized Orlicz norm $\norm{x}_{K_E} := \inf \set{ t > 0 \; ; \; \sum_{i=1}^n V_i(x_i/t) \leq E }$. 
Denote by $b_E := \int x \; d\lambda_{K_E}(x)$ the barycenter of $K_E$. Consider the set:
\begin{equation} \label{eq:intro-level-def}
\text{Level}(V) := \set{E \geq 0 \; ; \; e^{-E} \Vol(K_E) \geq \frac{1}{e} \frac{n^n e^{-n}}{n!}} .
\end{equation}
Then for all $ E \in \text{Level(V)} \cap (E_0 ,\infty)$:
\begin{align*}
 D_{\Poin}(\lambda_{K_E}) & \leq C \brac{M \log(e+ A^{(\infty,2)} M) + \frac{1}{\sqrt{n}} \int \abs{x} d\lambda_{K_E}(x)} \\
 & \leq C \brac{M \log(e+ A^{(\infty,2)} M) + \frac{\abs{b_E}}{\sqrt{n}} + D_{\Poin}^{\Lin}(\lambda_{K_E})} ,
\end{align*}
for an appropriate universal numeric constant $C > 1$.
\end{thm}

Since $D_{\Poin}(\lambda_{K_E})$ and $D_{\Poin}^{\Lin}(\lambda_{K_E})$ remain invariant under translation of $K_E$, by translating $V \mapsto V(\cdot + b)$ (i.e. $V_i \mapsto V_i(\cdot + b_i)$ for $b \in \Real^n$), we immediately obtain:

\begin{cor} \label{cor:main}
With the same notation and assumptions as in Theorem \ref{thm:main}, given $b \in \Real^n$ denote $A^{(\infty,2)}(b) := \frac{1}{\sqrt{n}} \snorm{(\alpha^{(\infty,2)}_i(b))_{i=1}^n}_2$, where:
 \begin{equation} \label{eq:alpha-infty}
\alpha^{(\infty,2)}_i(b) := (1+\norm{(V_i'(y) (y - b_i))_-}_{L^\infty(\mu_i)}) \vee \norm{V_i'(y) (y - b_i)}_{L^2(\mu_i)} . 
\end{equation}
Then for all $E \in \text{Level}(V)$ and $b \in \text{int}(K_E)$:
\[
 D_{\Poin}(\lambda_{K_E}) \leq C \brac{M \log(e+ A^{(\infty,2)}(b) M) + \frac{\abs{b_E-b}}{\sqrt{n}} + D_{\Poin}^{\Lin}(\lambda_{K_E})} .
 \]
\end{cor}

There is a particular value of $b \in \Real^n$ which is the most natural to use above -- the barycenter of $\mu := \exp(-V(x)) dx$, denoted:
\[
b_\mu := \int x \; d\mu(x) = \brac{\int x_i d\mu_i(x_i) }_{i=1}^n .
\]
With this choice, it is immediate to verify that $\set{\alpha_i(b_\mu)}$ are both scale and translation invariant in $\mu$. Another advantage we will verify in Lemma \ref{lem:Fradelizi2} is that when $b = b_\mu$, the $L^\infty$ term in (\ref{eq:alpha-infty}) is always majorized by the $L^2$ one; this is a generalization of the simple fact that $V_i'(y)(y - b_i) \geq 0$ whenever the minimum of the convex $V_i$ is attained at $b_i$. We consequently denote:
\[
A^{(2)}(b) := \frac{1}{\sqrt{n}} \snorm{(\alpha^{(2)}_i(b))_{i=1}^n}_2  ~,~  \alpha^{(2)}_i(b)  := \norm{V_i'(y) (y - b_i)}_{L^2(\mu_i)}  . 
\]

Let:
\[
\text{Cov}_\mu := \E((X_\mu -b_\mu) \otimes (X_\mu-b_\mu))  ~,~ \text{Cov}_E := \E((X_E - b_E) \otimes (X_E - b_E)) ,
\]
denote the corresponding covariance matrices, where $X_\mu$ and $X_E$ are distributed according to $\mu$ and $\lambda_{K_E}$, respectively. To provide some more relevant information regarding the subset $\text{Level}(V)$ of good levels $E$ and the associated level sets $K_E$, to which the above results apply, we have the following:
\begin{prop} \label{prop:Level}
Let $\mu = \exp(-V(x)) dx$ denote a log-concave probability measure on $\Real^n$ so that $\min V = 0$. For $E \geq 0$ let $K_E := \set{V \leq E}$, and let $\text{Level}(V)$ be defined by (\ref{eq:intro-level-def}). Let $b_\mu,b_E,\text{Cov}_\mu,\text{Cov}_E$ be defined as above. There exist numeric constants $c , C , C' > 0$ so that:
\begin{enumerate}
\item $\text{Level}(V)$ is a non-empty closed interval $[E_{\min},E_{\max}]$ with $E_{\min} \leq n$.
\item $1 \leq E_{\max} - E_{\min} \leq e \frac{n! e^n}{n^n} = e \sqrt{2 \pi n} (1 + o(1))$ as $n \rightarrow \infty$. 
\item $1+o(1) \leq \Vol(K_{E_{\min}})^{\frac{1}{n}} \leq \Vol(K_{E_{\max}})^{\frac{1}{n}} \leq e (1 + o(1))$ as $n \rightarrow \infty$.
\item $E_V := 1 + \int V(x) e^{-V(x)} dx$ satisfies $E_V \in \text{Level}(V)$ and $E_V \leq E_{\max} \wedge (n+1)$. 
\item \label{assertion:Fradelizi} $V(b_\mu) \leq E_V - 1 \leq (E_{\max} -1) \wedge n$, i.e. $b_\mu \in K_{E_V - 1} \subset K_{(E_{\max} -1) \wedge n}$. 
\item $\abs{b_E - b_\mu} \leq C \log (1+n) D_{\Poin}^{\Lin}(\mu)$, for all $E \in [E_{\min},E_{\max}]$.
\item $\text{Cov}_E \leq C' \log^2(1+n) \text{Cov}_\mu$ as positive-definite matrices, for all $E \in [E_{\min},E_{\max}]$.
\item $D_{\Poin}^{\Lin}(\lambda_{K_E}) \geq c > 0$ for all $E \geq E_{\min}$.
\end{enumerate}
\end{prop}

We remark that assertion (\ref{assertion:Fradelizi}) above is a refinement of a result of M.~Fradelizi \cite{FradeliziCentroid}, who showed that under the above assumptions $V(b_\mu) \leq n$ -- see Remark \ref{rem:Fradelizi} for further discussion. 
Combining Corollary \ref{cor:main} with Proposition \ref{prop:Level}, we can easily obtain:

\begin{thm}[Main Theorem]\label{thm:main2}
With the same notation and assumptions as in Corollary \ref{cor:main} and Proposition \ref{prop:Level}, for all $E \in [E_{\min},E_{\max}]$ such that $b = b_\mu \in \text{int}(K_E)$:
\[
D_{\Poin}(\lambda_{K_E}) \leq C_2 M \log(e + A^{(2)}(b_\mu) M) D_{\Poin}^{\Lin}(\lambda_{K_E})  . \]
In particular, this applies to all $E \in [(E_V - 1) \vee E_{\min}, E_{\max}]$, and notably, to $E = E_V$. 
\end{thm}
\smallskip
\noindent
Theorem \ref{thm:main2} confirms the KLS conjecture for $K_E$ as above whenever $A^{(2)}(b_\mu), M\leq C$.
By simultaneously rescaling all functions $V_i$, we can also easily remove the assumption that $\int \exp(-V_i(y)) dy = 1$ - see Corollary \ref{cor:main-scale}.

\smallskip

Finally, we state our $\log(1+n)$ estimate on the relation between the Poincar\'e constants of a general log-concave measure $\mu$ and its level-sets. 

\begin{thm}[From log-concave measure to good level-sets]\label{thm:logn}
Let $\mu = \exp(-V(x)) dx$ denote a log-concave probability measure on $\Real^n$ with $\min V = 0$. Denote its level sets by $K_E := \set{x \in \Real^n \; ;\; V(x) \leq E}$, and define as usual $\text{Level}(V)$ by (\ref{eq:intro-level-def}). Then for all $E \in \text{Level}(V)$, and in particular, for $E = E_V$, we have:
\[
D_{\Poin}(\lambda_{K_E}) \leq C D_{\Poin}(\mu) \log(e + \sqrt{n} D_{\Poin}(\mu)) .
\]
\end{thm}

\begin{rem}
An inspection of the proof of Theorem \ref{thm:logn} (and the relevant parts of Proposition \ref{prop:Level}) reveals that we could actually omit the $\sqrt{n}$ above, obtaining a dimension-independent estimate, for any $E \geq 0$ so that:
\begin{equation} \label{eq:exceptional}
e^{-E} \Vol(K_E) \geq c > 0,
\end{equation}
for some universal constant $c>0$. Unfortunately, such exceptionally good levels $E$ do not necessarily exist, and the best one can ensure in general is:
\[
\exists E \geq 0 \;\;\; e^{-E} \Vol(K_E) \geq \frac{n^n e^{-n}}{n!} = \frac{1}{\sqrt{2 \pi n}} (1 + o(1)) ,
\]
corresponding to the case $V(x) = \norm{x - x_0}$ for some norm $\norm{\cdot}$ (whose unit-ball has appropriate volume) and $x_0 \in \Real^n$, which results in the above $\sqrt{n}$ factor. As we did not find a reasonable condition for ensuring (\ref{eq:exceptional}), we only mention this variant in passing. 
\end{rem}

\section{Transferring Concentration: From Product Measure to Linearized One on Annulus} \label{sec:Lp}

Given a metric space $(X,d)$ and a Borel probability measure $\nu$, the associated concentration profile $\K = \K(X,d,\nu) : \Real_+ \rightarrow [0,1/2]$ is defined by:
\[
\K(r) := \sup \set{ \nu(X \setminus A^d_r) \; ; \; \mu(A) \geq 1/2} ~, ~ A^d_r := \set{x \in X \; ; \; d(x,A) < r} .
\]
Equivalently, it is well-known and immediate to verify that:
\[
\K(r) = \sup \set{ \nu \set{f \geq \med_\mu f + r} \; ; \; \text{$f : (X,d) \rightarrow \Real$ is $1$-Lipschitz }} ;
\]
here $\med_\nu f$ denotes any median of $f$ with respect to $\nu$, i.e. a median of the push-forward of $\nu$ by $f$. 

Given two Borel probability measures $\nu_1,\nu_2$ defined on $(X,d)$, we will require the following particular case of \cite[Proposition 2.2]{BartheEMilmanConservativeSpins} for transferring concentration information from $\K_1 = \K(X,d,\nu_1)$ to $\K_2 = \K(X,d,\nu_2)$. 

\begin{prop}[Barthe--Milman] \label{prop:BartheMilman}
Assume that $\snorm{\frac{d\nu_2}{d \nu_1}}_{L^p(\nu_1)} \leq  L$ for some $p \in (1,\infty]$.
Then setting $q = p^* = \frac{p}{p-1}$, we have:
\[
\K_2(r) \leq 2 L \K^{1/q}_1(r/2) \;\;\; \forall r > 0 . 
\]
\end{prop}

We will use Proposition \ref{prop:BartheMilman} with both $p <\infty$ and $p=\infty$. The latter simpler case, on which the proof of Proposition \ref{prop:BartheMilman} is in fact based, was originally proved in \cite[Lemma 3.1]{EMilmanGeometricApproachPartII} (with more precise numerical constants).

\subsection{Linearized Measure on Annulus}

Given a compact set $\Omega \subset \Real^n$ containing the origin in its interior, let $\norm{x}_\Omega := \inf \set{ \lambda > 0 \; ;\; x \in \lambda \Omega}$ denote its associated gauge function on $\Real^n$. 
Let $\mu = \exp(-V(x)) dx$ denote a general log-concave probability measure on $\Real^n$. Given $E > E_0 := V(0)$, denote by $K_E := \set{V \leq E}$ the convex level-set of $V$ at level $E$, which necessarily contains the origin in its interior. Let $\mu_{K_E}$ denote the probability measure on $\Real^n$ having density proportional to $\exp(-n \norm{x}_{K_E})$; it will be more convenient to write it as:
\[
\mu_{K_E} = \frac{1}{Z_E} e^{-(E + n(\norm{x}_{K_E} - 1))} dx ,
\]
where $Z_E > 0$ is a normalization constant ensuring that $\mu_{K_E}$ is a probability measure. Integration in polar coordinates easily yields:
\[
Z_E = \frac{n! e^n}{n^n} e^{-E} \Vol(K_E) .
\]

Given an additional parameter $w \in (0,1]$, we define the probability measure $\mu_{K_E,w}$ by conditioning $\mu_{K_E}$ on the annulus $1 - w \leq \norm{x}_{K_E} \leq 1$, namely:
\[
\mu_{K_E , w} := \frac{1}{Z_{E,w}} e^{-(E + n(\norm{x}_{K_E} - 1))} 1_{\norm{x}_{K_E} \in [1-w,1]} dx ,
\]
where again $Z_{E,w} > 0$ is an appropriate normalization constant. Note that the density of $\mu_{K_E,w}$ on the unit-sphere $\{\norm{x}_{K_E} = 1\}$ is constant and thus proportional to that of $\mu$. Furthermore, we will later see that our choice of the potential's slope (namely the coefficient $n$ above) coincides \emph{on-average} with that of $\mu$. Consequently, the measure $\mu_{K_E}$ should be thought of as a version of $\mu$ whose potential has been linearized about the unit-sphere $\{\norm{x}_{K_E} = 1\}$, with $\mu_{K_E,w}$ being in addition restricted to the annulus $\norm{x}_{K_E} \in [1-w,1]$. 

\begin{rem}
Our preference to work with the annulus $\norm{x}_{K_E} \in [1-w,1]$ instead of with (the perhaps more natural)  $\norm{x}_{K_E} \in [1-w,1+w]$, is because this permits us to employ a one-sided concentration estimate (Theorem \ref{thm:Hoeffding} below) instead of a two-sided one. Consequently, we only need to demand an $L^2$ integrability assumption from our random-variables, instead of an exponential integrability assumption which a standard two-sided estimate would require. 
\end{rem}

\smallskip

We will typically set $w = \frac{w_0}{\sqrt{n}}$ due to the following:

\begin{lem} \label{lem:Gamma}
For all $w_0 \in [0 , 1]$, if $w = \frac{w_0}{\sqrt{n}}$ then:
\[
Z_{E,w} \geq c w_0 Z_E ,
\]
for some universal numeric constant $c > 0$. 
\end{lem}
\begin{proof}
Let $X$ denote a random vector distributed according to $\mu_{K_E}$. Recalling that its density is proportional to $\exp(-n \norm{x}_{K_E})$, observe that $\norm{X}_{K_E}$ is distributed according to the Gamma distribution $\frac{n^n}{( n-1)!} e^{-nr} r^{n-1} dr$, and that:
\[
\frac{Z_{E,w}}{Z_E} = \P(\norm{x}_{K_E} \in [1-w,1]) = \frac{n^n}{( n-1)!} \int_{1-w}^{1} e^{-nr} r^{n-1} dr .
\]
The claim for $w_0$ of the order of $1$ is already clear, since the latter Gamma distribution may be realized as the law of $\frac{1}{n} \sum_{i=1}^n Y_i$, where $Y_i$ is a sequence of i.i.d. exponential random variables with parameter $1$, so that $\E(Y_i) = \Var(Y_i) = 1$ and hence $\E(\norm{X}_{K_E}) = 1$ and $\Var(\norm{X}_{K_E}) = \frac{1}{n}$; 
similarly, it is possible to extend this reasoning to all $w_0 \in [\frac{C}{\sqrt{n}},1]$ using the Berry--Esseen Theorem  (e.g. \cite{Petrov-SumsOfIndependentRVsBook}). 
To see the claim for all $w_0 \in [0,1]$, we use the fact that the density of the Gamma distribution is unimodal, and so we may lower bound the above integral as follows:
\[
\geq \frac{n^n}{( n-1)!} w \min\brac{e^{-n(1-w)} (1-w)^{n-1} , e^{-n} } . \]
Using Stirling's formula, we see that $w \frac{n^n}{( n-1)!} e^{-n} = w \sqrt{n} \frac{1}{\sqrt{2 \pi}} (1 + o(1))$ as $n \rightarrow \infty$, and in particular $\geq c' w_0$ for some constant $c' > 0$ and all $n \geq 1$. It remains to note that:
\[
e^{nw} (1 - w)^{n-1} \geq (1+ w)^{n} (1 - w)^{n-1} \geq (1-w^2)^{n-1} = \brac{1-\frac{w_0^2}{n}}^{n-1}\geq c'' > 0 ,
\]
for all $w_0 \in [0,1]$ and $n \geq 1$ (with $0^0$ interpreted as $1$). This concludes the proof. 

\end{proof}

\subsection{Dimension-Dependent Estimate}

Our proof of Theorem \ref{thm:logn} employs the following simple dimension-dependent estimate:

\begin{lem} \label{lem:Linfty}
For any log-concave probability measure $\mu = \exp(-V(x)) dx$, $E > V(0)$ and $w = \frac{w_0}{\sqrt{n}}$ with $w_0 \in (0,1]$, we have:
\[
\norm{\frac{d\mu_{K_E,w}}{d\mu}}_{L^\infty} \leq \frac{1}{ c w_0 Z_E} \exp(w_0 \sqrt{n}) ,
\]
where $c >0$ is the constant from Lemma \ref{lem:Gamma}. 
\end{lem}
\begin{proof}
Note that $V \leq E$ for all $\norm{x}_{K_E} \leq 1$, and in particular on the annulus $\norm{x}_{K_E} \in [1-w,1]$. It follows that on this annulus:
\[
\frac{d\mu_{K_E,w}}{d\mu}(x) = \frac{1}{Z_{E,w}} \exp(V - E - n (\norm{x}_{K_E} - 1)) \leq \frac{1}{Z_{E,w}} e^{n w} ,
\]
and the assertion follows by Lemma \ref{lem:Gamma}.
\end{proof}

\noindent
We will see in Section \ref{sec:proofs} that this is already enough to deduce the worst-case $\log(1+n)$ estimate of Theorem \ref{thm:intro}.

\subsection{Dimension-Independent Estimate}

To obtain a dimension-independent estimate, we restrict ourselves in this subsection to product measures $\mu$. Let $V_i : \Real\rightarrow \Real$, $i=1,\ldots,n$, denote a sequence of convex functions so that $\mu_i := \exp(-V_i(y)) dy$ is a probability measure on $\Real$. Denote $V(x) := \sum_{i=1}^n V_i(x_i)$, $x \in \Real^n$, and let $\mu$ denote the corresponding product measure on $\Real^n$:
\[
\mu := \mu_1 \otimes \ldots \otimes \mu_n = \exp(-V(x)) dx .
\]

Our proof of Theorem \ref{thm:main} relies on the following crucial estimate. The proof strategy is in some sense similar to the one employed in \cite{BartheEMilmanConservativeSpins}, where a zeroth order approximation was used about a hyperplane (instead of a first order approximation about a convex hypersurface as in the present case).
\begin{prop} \label{prop:Lp}
Let $(\alpha^{(\infty,2)}_i)_{i=1}^n$ be defined as in (\ref{eq:main-psi1}), and recall that $A^{\infty,2} := \frac{1}{\sqrt{n}}\snorm{(\alpha^{(\infty,2)}_i)_{i=1}^n}_2$.
Let $w_0 \in (0,1/2]$, $w = \frac{w_0}{\sqrt{n}}$ and $p \geq 1$. 
Then for all $E > E_0 := V(0)$:
\[
\int \brac{\frac{d\mu_{K_E,w}}{d\mu}}^p d\mu \leq \frac{1 + \sqrt{2 \pi}}{Z_E^p} \frac{1}{(c w_0)^p} \exp(8 p^2 w_0^2 (A^{(\infty,2)})^2) ,
\]
where $c > 0$ is the constant from Lemma \ref{lem:Gamma}.
\end{prop}

For the proof, we will require the following concentration-inequality for sums of independent random-variables.
When the random-variables are bounded, this inequality is classical and due to Hoeffding \cite{Hoeffding-ClassicalInq}; we will need the following version, when the random-variables are only assumed to be bounded from one side (see Maurer \cite[Theorem 2.1, Corollary 2.2]{Maurer-OneSidedHoeffding} for a simple derivation, Bentkus \cite[Theorem 1.3]{Bentkus-OneSidedHoeffding} for improved optimal constants in the exponent; compare also with an earlier result by McDiarmid \cite[Theorem 2.7]{McDiarmid-ConcentrationChapter} in the spirit of Bernstein's inequality \cite[Corollary 2.11]{BLM-Book}):

\begin{thm}[One-Sided Hoeffding Inequality] \label{thm:Hoeffding}
Let $Y_1,\ldots,Y_n$ denote a sequence of independent random variables so that $\E(Y_i) = 0$ and:
\[ \norm{(Y_i)_{-}}_{L^\infty} \vee \norm{Y_i}_{L^2} < \infty ~~\forall i=1,\ldots,n .
\] Then:
\[
\P\brac{\sum_{i=1}^n Y_i \leq -r } \leq \exp\brac{- \frac{1}{2} \frac{r^2}{\sum_{i=1}^n  (\norm{(Y_i)_{-}}^2_{L^\infty} + \norm{Y_i}^2_{L^2})}} \;\;\; \forall r > 0  .
\]
\end{thm}
\noindent
Here $\norm{Z}_{L^2} := (\mathbb{E} Z^2)^{1/2}$ and $\norm{Z}_{L^\infty} := \esssup \abs{Z}$.

\smallskip

\begin{proof}[Proof of Proposition \ref{prop:Lp}]
Recall that $K_E$ was defined as the convex level set $\set{V \leq E}$, and so for all $x \neq 0$, $V(\frac{x}{1+t}) = E$ with $t := \norm{x}_{K_E} - 1$. 
Consequently, for all $r \geq 0$:
\begin{align*}
g(r) & := \mu \set{ \norm{x}_{K_E} \in [1-w,1] \; ; \; V(x) - (E + n (\norm{x}_{K_E} - 1)) \geq r} \\
& \leq \mu \set{x \in \Real^n \; ; \; \exists t \in [-w,0] \;,\; V(x) - \brac{V\brac{\frac{x}{1+t}} + n t}\geq r} .
\end{align*}
By convexity of $V$, we know that on the subset of full measure in $\Real^n$ where $V$ is differentiable we have:
\[
V\brac{\frac{x}{1+t}} = V\brac{x - \frac{t}{1+t} x } \geq V(x) - \frac{t}{1+t} \scalar{\nabla V(x), x} ,
\]
and so we may continue the chain of inequalities above as follows (as $1+t \geq 0$):
\begin{align}
\nonumber & \leq \mu \set{x \in \Real^n \; ; \; \exists t \in [-w,0]  \;,\; \frac{t}{1+t} \scalar{\nabla V(x),x} - n t \geq r} \\
\nonumber & \leq \mu \set{x \in \Real^n \; ; \; \exists t \in [-w,0]  \;,\;  t\scalar{\nabla V(x),x} - n t \geq r(1+t) + n t^2 } \\ 
\label{eq:annulus1} & \leq  \mu \set{x \in \Real^n \; ; \; \scalar{\nabla V(x),x} - n \leq - r \frac{1-w}{w} } .
\end{align}

Note that if $X$ is a random-vector distributed according to $\mu$, then $\scalar{\nabla V(X),X} - n = \sum_{i=1}^n Y_i$, where $Y_i := V_i'(X_i) X_i - 1$ are independent random-variables with each $X_i$ distributed according to $\mu_i$. 
Integrating by parts, we clearly have:
\[
\E Y_i = \int x_i V_i'(x_i) \exp(-V_i(x_i)) dx_i - 1 = 0 ,
\]
and our assumption (\ref{eq:main-psi1}) translates into:
\[
\norm{(Y_i)_-}_{L^\infty} \vee \norm{Y_i}_{L^2}  \leq \brac{1 + \norm{(V_i'(y) y)_-}_{L^\infty(\mu_i)}} \vee \norm{V_i'(y) y}_{L^2(\mu_i)} = \alpha^{(\infty,2)}_i < \infty . 
\]
Applying Theorem \ref{thm:Hoeffding}, we deduce:
\begin{equation} \label{eq:crucial}
\P\brac{\sum_{i=1}^n Y_i \leq -\sqrt{n} s} \leq \exp\brac{-\frac{1}{4} \frac{s^2}{(A^{(\infty,2)})^2} }  \;\;\; \forall s > 0 .  
\end{equation}
Since we assume that $w = \frac{w_0}{\sqrt{n}} \leq \frac{1}{2}$, we have $\frac{1-w}{w} \geq \frac{1}{2w}$, so we apply the above inequality with $s = \frac{r}{2 w_0}$ to estimate (\ref{eq:annulus1}), and deduce that:
\[
g(r) \leq \exp\brac{- \frac{r^2}{16 w_0^2 (A^{(\infty,2)})^2} } \;\;\; \forall r > 0 . \]

Integrating by parts and using the elementary inequality $x \leq \exp(x^2/2)$, 
we can now deduce:
\begin{align*}
\int \brac{\frac{d\mu_{K_E,w}}{d\mu}}^p d\mu &= \frac{1}{Z_{E,w}^p} \int \exp\brac{p \brac{V(x) - \brac{E + n (\norm{x}_{K_E} - 1)}}} 1_{\norm{x}_{K_E} \in [1-w, 1]}  d\mu(x) \\
 & \leq \frac{1}{Z_{E,w}^p} \brac{1 + \int_0^\infty p \exp(p r) g(r) dr } \\
 & \leq  \frac{1}{Z_{E,w}^p} \brac{1 + \sqrt{2 \pi} \sqrt{8 p^2 w_0^2 (A^{(\infty,2)})^2} \exp( 4 p^2 w_0^2 (A^{(\infty,2)})^2)  } \\
 & \leq \frac{1}{Z_{E,w}^p} \brac{1 + \sqrt{2 \pi} \exp(8 p^2 w_0^2 (A^{(\infty,2)})^2) } \\
 & \leq \frac{1 + \sqrt{2 \pi}}{Z_{E,w}^p} \exp(8 p^2 w_0^2 (A^{(\infty,2)})^2) .
\end{align*}
The asserted estimate then follows by recalling Lemma \ref{lem:Gamma}. 
\end{proof}

\begin{rem}
The proof presented above is essentially the only place in this work where the product structure of $\mu = \exp(-V(x)) dx$ (or equivalently, the separable structure of $V(x) = \sum_{i=1}^n V_i(x_i)$) is used. The sole purpose of this product structure is to obtain (\ref{eq:crucial}), asserting a strong concentration of $\scalar{\nabla V(X),X}$ around its mean $n$. Any other condition which ensures a similar strong concentration would equally result in confirmation of the KLS conjecture for additional classes of convex bodies, simply by following the arguments in this work. 
\end{rem}

\section{Properties of the level-set $K_E$} \label{sec:ZE}

\subsection{Bounding $\Vol(K_E)$}

To apply the bound of Proposition \ref{prop:Lp}, we will need to bound $Z_E$ from below, where recall:
\[
Z_E = \frac{n! e^n}{n^n} e^{-E} \Vol(K_E) ~,~ K_E = \set{V \leq E} .
\]

\begin{dfn*}
Given $q \geq 0$, denote:
\begin{equation} \label{eq:level-def}
\text{Level}^{(q)}(V) := \set{E \geq  0 \; ; \; e^{-E} \Vol(K_E) \geq e^{-q} \frac{n^n e^{-n}}{n!}} = \set{E \geq 0 \; ; \; Z_E \geq e^{-q}} ,
\end{equation}
the subset of good level-sets of $V$. 
\end{dfn*}

\noindent
Consequently, our goal will be to study $\text{Level}^{(q)}(V)$; obviously this family is monotone increasing in $q$. 
Note that $\text{Level}(V)$ defined in Section \ref{sec:results} is precisely $\text{Level}^{(1)}(V)$. All of our results in this work remain valid with $\text{Level}^{(1)}(V)$ replaced by $\text{Level}^{(q)}(V)$ for any $q \geq 1$, with an additional appropriate dependence on $q$, but with the exception of this section, we refrain from this extraneous generality. 

\smallskip
Now observe that the convexity of $V$ ensures that the map:
\begin{equation} \label{eq:def-g}
\Real_+ \ni E \mapsto g(E) := \Vol(K_E)^{\frac{1}{n}} 
\end{equation}
is a concave function on its support by the Brunn-Minkowski inequality (e.g. \cite{Schneider-Book,GardnerSurveyInBAMS}). By separability $\min V = \sum_{i=1}^n \min V_i = 0$, and so $g$ is supported on $\Real_+$. Note that we do not assume that $g(0) = 0$. Lastly, integration by parts yields:
\begin{equation} \label{eq:normalization}
\int_0^\infty e^{-E} g(E)^n dE = \int_0^\infty e^{-E} \Vol\set{V \leq E} dE = \int_{\Real^n} e^{-V(x)} dx = \int_{\Real^n} d\mu = 1 . 
\end{equation}
On the basis of these three properties, we will prove a slightly more detailed version of assertions (1) - (3) of Proposition \ref{prop:Level} from the Introduction:
\begin{prop} \label{prop:ZE}
Let $\mu = \exp(-V(x)) dx$ denote a log-concave probability measure on $\Real^n$ with $\min V = 0$. For $E \geq 0$ let $K_E := \set{V \leq E}$, and let $\text{Level}^{(q)}(V)$ be defined by (\ref{eq:level-def}) for $q \geq 0$. Then:
\begin{enumerate}
\item $\text{Level}^{(q)}(V)$ is a non-empty closed interval $[E^{(q)}_{\min},E^{(q)}_{\max}]$ with $E^{(q)}_{\min} \leq n$.
\item $q \leq E^{(q)}_{\max} - E^{(q)}_{\min} \leq e^q \frac{n! e^n}{n^n} = e^q \sqrt{2 \pi n} (1 + o(1))$ as $n \rightarrow \infty$. 
\item Denoting $c^{(q)}_n := e^{-\frac{q}{n}}\frac{n/e}{(n!)^{1/n}} \rightarrow 1$ as $n \rightarrow \infty$, we have:
\[
c^{(q)}_n \leq \Vol(K_{E^{(q)}_{\min}})^{\frac{1}{n}} \leq 1 \vee e c^{(q)}_n  ~,~  c^{(q)}_n \leq  \Vol(K_{E^{(q)}_{\max}})^{\frac{1}{n}} \leq e c^{(q)}_n (1+o(1)).
\]
\end{enumerate}
\end{prop}

\noindent
The proof is a based on the following:

\begin{lem} \label{lem:concave}
Let $g : \Real_+ \rightarrow \Real_+$ denote a (non-decreasing) non-negative concave function, so that:
\[
\int_0^\infty e^{-t} g(t)^n dt = 1 .
\]
Let $M_g$ denote the maximum of $e^{-t} g(t)^n$ on $\Real_+$, and let $t_g > 0$ denote the (necessarily unique) point on which it is attained. 
Then $t_g \leq n$ and $M_g \geq e^{-n} \frac{n^n}{n!}$.
\end{lem}
\begin{rem} \label{rem:concave}
Observe that both asserted estimates are sharp for the model function $g_0(t) := \frac{t}{(n!)^{1/n}}$, which indeed satisfies $\int_0^\infty e^{-t} g_0(t)^n dt = 1$.
\end{rem}
\begin{rem}
There are numerous instances in the literature of similar looking lemmas regarding measures on $\Real_+$ of the form $e^{g(t)} t^{n-1} dt$ with $g$ concave (and typically decreasing), arising when integrating a log-concave measure in $\Real^{n}$ in polar coordinates (see e.g. \cite[Lemma 2.1 and 2.2]{KlartagMilmanLogConcave}). However, we emphasize that it is not possible to obtain the delicate lower bound we need on $e^{-E} \Vol(K_E)$ by integrating $\mu = \exp(-V(x)) dx$ in polar coordinates - it is not hard to check that there is no fixed value of $E$ (including the typical guess $E=n$) that will in general work for every ray simultaneously.
\end{rem}
\begin{proof}[Proof of Lemma \ref{lem:concave}]
First, note that $t_g$ is indeed unique since $e^{t/n}$ is not concave on any non-empty open interval. For simplicity, we may assume (by approximation) that $g$ is differentiable. The maximum of $t \mapsto h(t) := -t + n \log g(t)$ on $\Real_+$ is attained at $t_g$, and hence either $t_g = 0$ or $g'(t_g) = \frac{g(t_g)}{n}$. In the latter case, concavity implies that $g'(t_g) \leq \frac{g(t_g) - g(0)}{t_g} \leq \frac{g(t_g)}{t_g}$, and therefore $t_g \leq n$. In the former case, necessarily $h'(0) \leq 0$, i.e.  $g'(0) \leq \frac{g(0)}{n}$. In either case, concavity implies:
\[
 g(t_g + s) \leq g(t_g) + g'(t_g) s \leq g(t_g) \brac{1 + \frac{s}{n}} \;\;\; \forall s \in [-t_g,\infty) .
\]
Consequently:
\begin{align*}
1 &= \int_0^\infty e^{-t} g(t)^n dt \leq M_g \int_{-t_g}^\infty e^{-s} \brac{1 + \frac{s}{n}}^n ds \leq M_g \int_{-n}^\infty e^{-s} \brac{1 + \frac{s}{n}}^n ds \\
& = M_g e^n \int_0^\infty e^{-t} \brac{\frac{t}{n}}^n dt = M_g e^n \frac{n!}{n^n} ,
\end{align*}
concluding the proof. 
\end{proof}

In fact, although this will not be used anywhere else in this work, we can claim the following:
\begin{prop} \label{prop:contraction}
With the same assumptions and notation as in Lemma \ref{lem:concave} and Remark \ref{rem:concave}, there exists an increasing and contracting map $T: \Real_+ \rightarrow \Real_+$ so that $T$ pushes forward the probability measure $e^{-t} g_0(t)^n dt$ onto $e^{-t} g(t)^n dt$ on $\Real_+$.
In particular, we have $t'_g := T(n) \leq n$ and $e^{-t'_g} g(t'_g)^n \geq e^{-n} g_0(n)^n = e^{-n} \frac{n^n}{n!}$. 
\end{prop}

This should be compared with a well-known contraction result established by L.~Caffarelli in \cite{CaffarelliContraction} (see also \cite{KimEMilmanGeneralizedCaffarelli} for generalizations), asserting that the Brenier optimal-transport map $T$ pushing forward a Gaussian probability measure $\gamma_n$ on $\Real^n$ onto a probability measure $\mu = \exp(-V) \gamma_n$ with $V : \Real^n \rightarrow \Real$ convex, is in fact contracting Euclidean distance. While we do not know how to extend Proposition \ref{prop:contraction} to higher-dimension, we can obtain contraction results between members of an entire family of one-dimensional model spaces and appropriate concave perturbations thereof, including both the Gaussian measure and the measure $e^{-t} g_0(t)^n dt$ as particular cases. As this is too off-topic for this work, these contraction results, as well as the proof of Proposition \ref{prop:contraction}, will appear elsewhere. 

\smallskip

Let us now complete the proof of Proposition \ref{prop:ZE}:

\begin{proof}[Proof of Proposition \ref{prop:ZE}]
Recall the definition (\ref{eq:def-g}) of the function $g$. 
Lemma \ref{lem:concave} implies that $t_g \in \text{Level}^{(q)}(V)$ for all $q \geq 0$. As the function $E \mapsto e^{-E} g(E)^n$ is integrable, continuous and unimodal (as its logarithm is concave), $\text{Level}^{(q)}(V)$ must be a closed interval $[E^{(q)}_{\min},E^{(q)}_{\max}] \subset [0,\infty)$. Lemma \ref{lem:concave} implies that $E^{(q)}_{\min} \leq t_g \leq n$. We also have:
\[
1 = \int_0^\infty e^{-E} g(E)^n dE \geq \int_{E^{(q)}_{\min}}^{E^{(q)}_{\max}} e^{-E} g(E)^n dE \geq (E^{(q)}_{\max} - E^{(q)}_{\min}) e^{-q} \frac{n^n e^{-n}}{n!} ,
\]
implying that $E^{(q)}_{\max} - E^{(q)}_{\min} \leq e^q \frac{n! e^n}{n^n} = e^q \sqrt{2 \pi n} (1 + o(1))$ by Stirling's formula. To see the reverse inequality, observe that $W(t) := t - n \log g(t)$ satisfies $W'(t) \leq 1$, and therefore:
\[
e^{-W(t_g + q)} \geq e^{-q} e^{-W(t_g)} \geq e^{-q} \frac{n^n e^{-n}}{n!}.
\]
Consequently $t_g + q \in \text{Level}^{(q)}(V)$, implying $E^{(q)}_{\max} - E^{(q)}_{\min} \geq q$. Clearly:
\[
\Vol(K_{E^{(q)}_{\min}})^{\frac{1}{n}} \geq e^{\frac{E^{(q)}_{\min}}{n}} e^{-\frac{q}{n}} c_n ,
\]
with $c_n = \frac{n/e}{(n!)^{1/n}} \rightarrow 1$ by Stirling's formula, where the only possible strict inequality above is when $E^{(q)}_{\min} = 0$, in which case $\Vol(K_{0}) \leq 1$ (as $\int \exp(-V) dx = 1$). Since $E^{(q)}_{\min} \in [0,n]$, the upper and lower estimates on $\Vol(K_{E^{(q)}_{\min}})^{\frac{1}{n}}$ follow. Finally, note that:
\[
\Vol(K_{E^{(q)}_{\max}})^{\frac{1}{n}} = e^{\frac{E^{(q)}_{\max}}{n}} e^{-\frac{q}{n}} c_n = e^{\frac{E^{(q)}_{\min}}{n}} e^{-\frac{q}{n}} c_n  e^{\frac{E^{(q)}_{\max} - E^{(q)}_{\min}}{n}} ,
\]
and so the upper estimate on $\Vol(K_{E^{(q)}_{\max}})^{\frac{1}{n}}$ follows since $(E^{(q)}_{\max} - E^{(q)}_{\min})/n \leq e^q \sqrt{2 \pi/n} (1+o(1)) = o(1)$ as $n \rightarrow \infty$. 
This completes the proof.
\end{proof}

\subsection{Barycenter of $\mu$}

We now turn to prove assertions (4) and (\ref{assertion:Fradelizi}) of Proposition \ref{prop:Level}, which we equivalently reformulate as follows:
\begin{prop}[Refinement of Fradelizi's Bound] \label{prop:Fradelizi}
Let $\mu = \exp(-V(x)) dx$ denote a log-concave probability measure on $\Real^n$ with $\min V = 0$. Let $b_\mu := \int x \; d\mu(x)$ denote the barycenter of $\mu$, and set:
\[
E_V := 1 + \int V(x) e^{-V(x)} dx . 
\]
Then:
\begin{enumerate}
\item $E_V \in \text{Level}^{(1)}(V)$ and $E_V \leq E_{\max}^{(1)} \wedge (n+1)$. 
\item $V(b_\mu) \leq E_V - 1 \leq (E_{\max}^{(1)} - 1) \wedge n$, i.e. $b_\mu \in K_{E_V - 1} \subset K_{(E^{(1)}_{\max} - 1) \wedge n}$.
\end{enumerate}
\end{prop}

\begin{rem} \label{rem:Fradelizi}
In \cite{FradeliziCentroid}, M.~Fradelizi showed that for any log-concave probability measure $\mu = \exp(-V(x)) dx$ on $\Real^n$,  $V(b_\mu) \leq \min V + n$; we will present a simplified proof of this bound below. While this is sharp whenever $\mu$ is log-affine on an appropriate convex cone, it is easy to construct (non-trivial) examples when this estimate can be significantly improved. For instance, let $V(x) = \norm{x}^p_K$ for any convex body $K$ having the origin in its interior and $p\geq 1$ (and $K$ is scaled so that $\mu$ is a probability measure). In that case, it is immediate to check that $g(E) := \Vol(K_E)^{1/n} = c_{n,p}  E^{1/p}$ (with $c_{n,p} = \Gamma(n/p+1)^{-1/n}$), and we have $t_g := \argmax_E e^{-E} \Vol(K_E) =  n/p$ with the notation of Lemma \ref{lem:concave}. As $t_g \in [E^{(1)}_{\min},E^{(1)}_{\max}]$ and $E^{(1)}_{\max} - E^{(1)}_{\min}\leq e \sqrt{2 \pi n} (1+o(1))$, we conclude that $V(b_\mu) \leq E^{(1)}_{\max} - 1 = \frac{n}{p} (1 +o(1))$ as $n \rightarrow \infty$, yielding a strict improvement over Fradelizi's estimate for any fixed $p > 1$ and large enough $n$. 
\end{rem}

\begin{proof}[Proof of Proposition \ref{prop:Fradelizi}]
We may assume by translating $\mu$ if necessary that the minimum of $V$ as attained at the origin. It follows by Jensen's inequality, convexity of $V$ and integration by parts, that:
\begin{align*}
V(b_\mu) & \leq \int V(x) e^{-V(x)} dx = V(0) + \int (V(x) - V(0)) e^{-V(x)} dx \\
& \leq V(0) + \int \scalar{\nabla V(x),x} e^{-V(x)} dx = V(0) + \int \text{div}(x) e^{-V(x)} dx = V(0) + n ,
\end{align*}
immediately recovering Fradelizi's bound. Recalling our assumption that $\min V = 0$, we have verified that:
\[
V(b_\mu) \leq E_V -1 \leq n . 
\]
It remains to show that $E_V \in \text{Level}^{(1)}(V)$. Recall our notation $g(E) := \Vol(K_E)^{1/n}$, and introduce the following measure on $\Real_+$:
\[
\nu := e^{-W(E)} dE = e^{-E} g(E)^n dE .
\]
Also recall that by (\ref{eq:normalization}) $\nu$ is a probability measure, and as $g$ is concave $W$ is in particular convex and hence $\nu$ is log-concave. 
Integrating by parts on the distribution of $\Vol \set{V \leq E}$, we obtain:
\begin{align*}
& E_V -1 = \int V(x) e^{-V(x)} dx = -\int_0^\infty \frac{d}{dE} (E e^{-E}) \Vol \set{V \leq E} dE \\
&= \int_0^\infty e^{-E}(E-1) \Vol(K_E) dE = \int_0^\infty E e^{-E} g(E)^n dE - 1 ,
\end{align*}
thereby concluding that $E_V$ coincides with the barycenter of $\nu$:
\[
E_V = b_\nu := \int_0^\infty E d\nu(E) . 
\]
Applying Fradelizi's bound in the one-dimensional case, we know that $W(b_\nu) \leq W(t_g) + 1$, where recalling the notation of Lemma \ref{lem:concave}, $t_g$ is the maximum point of $e^{-W}$. Invoking Lemma \ref{lem:concave}, we obtain:
\[
e^{-b_\nu} \Vol(K_{b_\nu}) = e^{-W(b_\nu)} \geq \frac{1}{e} e^{-W(t_g)} \geq \frac{1}{e}\frac{n^n e^{-n}}{n!} .
\]
It follows by definition that $E_V = b_\nu \in \text{Level}^{(1)}(V)$ (and in particular $E_V \leq E^{(1)}_{\max}$), thereby concluding the proof. 

\end{proof}

In addition, we will require the following:
\begin{lem} \label{lem:Fradelizi2}
Let $\nu = \exp(-W(y)) dy$ denote a log-concave probability measure on $\Real$. Then:
\begin{enumerate}
\item $\norm{W'(y) (y-b)}_{L^2(\nu)} \geq \sqrt{2}$, for all $b \in \Real$. 
\item $\norm{(W'(y) (y-b_\nu))_{-}}_{L^\infty(\nu)} \leq 1$, where $b_\nu$ denotes the barycenter of $\nu$. 
\end{enumerate}
\end{lem}
\begin{proof}
For the first assertion, we may assume by a standard approximation argument that $W$ is $C^2$ smooth. Using $W'' \geq 0$ and integrating by parts, we verify that:
\begin{align*}
& \int (y-b)^2 W'(y)^2 \exp(-W(y)) dy \geq \int (y-b)^2 (W'(y)^2 - W''(y)) \exp(-W(y)) dy \\
& = \int (y-b)^2 (\exp(-W(y)))'' dy = \int 2 \exp(-W(y)) dy = 2 . 
\end{align*}
For the second assertion, note that by convexity, for any $b \in \Real$:
\[
W'(y) (y-b) \geq W(y) - W(b) \geq \min W - W(b) . 
\]
On the other hand, Fradelizi's estimate (Remark \ref{rem:Fradelizi}) in the one-dimensional case asserts that $W(b_\nu) \leq \min W + 1$, thereby concluding the proof. 
\end{proof}

\subsection{Barycenter and covariance matrix of $K_E$}

We conclude this section by providing a proof of assertions (6) and (7) of Proposition \ref{prop:Level}; assertion (8) will be proved in Section \ref{sec:proofs}. Recall that $X_\mu$ and $X_{E}$ are assumed to be distributed according to $\mu$ and $\lambda_{K_E}$, respectively, and that we denote the corresponding barycenters:
\[
b_\mu := \E(X_\mu) ~,~ b_E := \E(X_{E}) ,
\]
and covariance matrices:
\[
\text{Cov}_\mu := \E((X_\mu -b_\mu) \otimes (X_\mu-b_\mu))  ~,~ \text{Cov}_E := \E((X_E - b_E) \otimes (X_E - b_E)) .
\]
Note that by definition:
\[
D_{\Poin}^{\Lin}(\mu) = \max_{\theta \in S^{n-1}} \sqrt{\scalar{\text{Cov}_\mu \; \theta,\theta}} , \]
where $S^{n-1}$ denotes the Euclidean unit-sphere in $(\Real^n,\abs{\cdot})$.

\begin{prop}
For all $q \geq 0$ and $E \in [E^{(q)}_{\min},E^{(q)}_{\max}]$:
\begin{enumerate}
\item $\abs{b_E - b_\mu} \leq C (1+q) \log (1+n) D_{\Poin}^{\Lin}(\mu)$. 
\item $\text{Cov}_E \leq C' (1+q)^2 \log^2(1+n) \text{Cov}_\mu$ as positive-definite matrices.
\end{enumerate}
Here $C,C'>0$ are two universal numeric constants. 
\end{prop}
\begin{proof}
Assume by translating $\mu$ if necessary that $b_\mu = 0$. For any $\theta \in S^{n-1}$, a well-known consequence of Borell's lemma \cite{Borell-logconcave} (cf. \cite[Appendix A]{SpanishBook}) is that:
\[
\mu(\scalar{x,\theta} \geq t) \leq C \exp(- t / (C L_\theta)) \;\;\; \forall t \in \Real ~,
\]
where $L_\theta := \sqrt{\scalar{\text{Cov}_\mu \; \theta,\theta}}$ and $C > 0$ is a numeric constant. On the other hand:
\[
\mu(\scalar{x,\theta} \geq t) \geq e^{-E} \Vol(K_E) \lambda_{K_E}(\scalar{x,\theta} \geq t) \geq \frac{1}{e^q} \frac{n^n e^{-n}}{n!} \lambda_{K_E}(\scalar{x,\theta} \geq t) ,
\]
for any $E \in [E_{\min}^{(q)},E_{\max}^{(q)}]$. 
Applying this to $t = t_\theta := \scalar{b_E,\theta}$, note that:
\[
\lambda_{K_E}(\scalar{x,\theta} \geq t_\theta) = \lambda_{K_E}(\scalar{x - b_E,\theta} \geq 0) \geq \brac{\frac{n}{n+1}}^n \geq \frac{1}{e} ,
\]
by Gr\"{u}nbaum's Theorem \cite{GrunbaumSymmetry} (cf. \cite{Fradelizi-Habilitation}) on the volume of halfspaces passing through the barycenter of a convex body. Combining everything, we obtain by Stirling's formula:
\[
C \exp(-\scalar{b_E,\theta} / (C L_{\theta})) \geq \mu(\scalar{x,\theta} \geq t_\theta) \geq \frac{1}{e^{q+1} \sqrt{2 \pi n}} (1+o(1)) ,
\]
as $n \rightarrow \infty$. In particular, it follows (as $L_\theta = L_{-\theta}$) that:
\begin{equation} \label{eq:strong-bar}
\abs{\scalar{b_E,\theta}}\leq C L_{\theta} \brac{C_1 + q + \frac{1}{2} \log n}  
\leq C_2 (1+q) \log(1+n) D_{\Poin}^{\Lin}(\mu) ,
\end{equation}
establishing (in fact, a strengthening of) the first assertion. 

The second assertion is proved similarly. Indeed, for all $\theta \in S^{n-1}$, $t \in \Real$ and $E \in [E_{\min}^{(q)},E_{\max}^{(q)}]$:
\begin{align*}
\frac{1}{e^q} \frac{n^n e^{-n}}{n!} & \lambda_{K_E}(\scalar{x-b_E,\theta} \geq t) \leq e^{-E} \Vol(K_E) \lambda_{K_E}(\scalar{x-b_E,\theta} \geq t) \\
& \leq \mu(\scalar{x,\theta} \geq t + \scalar{b_E ,\theta}) \leq C \exp(-(t+ \scalar{b_E ,\theta}) / (C L_\theta)) .
\end{align*}
Invoking (\ref{eq:strong-bar}) and applying Stirling's formula again, we deduce that for all $\theta \in S^{n-1}$ and $t > 0$:
\[
\lambda_{K_E}(\abs{\scalar{x-b_E,\theta}} \geq t) \leq 2 (1+n)^{C_2(1+q)} \exp(- t / (C L_\theta)) . 
\]
Integrating by parts, it follows that:
\[
\scalar{\text{Cov}_{E} \theta,\theta} = \int \scalar{x-b_E,\theta}^2 d\lambda_{K_E}(x) \leq \int_0^\infty 2t \min(1,2 (1+n)^{C_2(1+q)} \exp(- t / (C L_\theta))) dt ,
\]
and the latter integral is easily seen to be bounded above by $C_3 (1+q)^2 \log^2(1+n) L^2_\theta$, thereby concluding the proof. 
\end{proof}

\section{Transferring Concentration: From Annulus to Cone and Uniform Measures} \label{sec:W1}

Given two Borel probability measures $\mu_1,\mu_2$ on a common metric space $(X,d)$, recall that their $1$-Wasserstein distance $W_{d,1}(\mu_1,\mu_2)$ is defined as:
\[
W_{d,1}(\mu_1,\mu_2) := \inf_\pi \int d(x,y) d\pi(x,y) ,
\]
where the infimum is over all Borel probability measures $\pi$ on $X \times X$ having first and second marginals $\mu_1$ and $\mu_2$, respectively. 
By the Monge--Kantorovich--Rubinstein dual characterization of $W_{d,1}$ (e.g. \cite[Case 5.16]{VillaniOldAndNew}), we have:
\[
W_{d,1}(\mu_1 , \mu_2) = \sup \set{ \int f (d\mu_1 - d\mu_2) \; ; \; f : (X,d) \rightarrow \Real \text{ is $1$-Lipschitz} } .
\]
The following immediate consequence of this dual characterization was first noted in \cite[Lemma 5.4]{EMilmanGeometricApproachPartII}, allowing transferring first-moment concentration of Lipschitz functions between two measures which are close in $W_{d,1}$-distance:

\begin{lem}[\cite{EMilmanGeometricApproachPartII}] \label{lem:W1}
For any $1$-Lipschitz function $f$ on $(X,d)$, we have:
\[
\int \abs{f - \med_{\mu_2} f} d\mu_2 \leq \int \abs{f - \med_{\mu_1} f} d\mu_1 + W_{d,1}(\mu_1,\mu_2) . 
\]
\end{lem}

\noindent
Here $\med_{\nu} f \in \Real$ denotes a median of $f$ under the law of (the probability measure) $\nu$, i.e. a median of the probability measure $f_* \nu$ on $\Real$.

\subsection{From Annulus to Cone Measure}

Let $\Omega \subset \Real^n$ denote a compact set containing the origin in its interior and having Lipschitz boundary. We will say that $\Omega$ is a star-shaped body if in addition it contains all intervals adjoining its elements to the origin. Recall that $\norm{x}_\Omega$ denotes the gauge function of $\Omega$. 
We denote by $\sigma_{\partial \Omega}$ the induced cone probability measure on $\partial \Omega$, i.e. the push-forward of $\lambda_{\Omega}$ via the map $x \mapsto \frac{x}{\norm{x}_\Omega}$.
It is well-known and immediate to check that:
\[
\sigma_{\partial \Omega} = \frac{1}{\Vol(\Omega)} \frac{\scalar{ x, \nu}}{n} \cdot \H^{n-1}|_{\partial \Omega} ,
\]
where $\H^{n-1}$ denotes the $n-1$-dimensional Hausdorff measure in Euclidean space $(\Real^n,\abs{\cdot})$, and $\nu$ denotes the ($\H^{n-1}$-a.e. defined) outer unit-normal to $\partial \Omega$. 

\begin{lem} \label{lem:W1-estimate}
Let $\Omega \subset \Real^n$ denote a star-shaped body, and let $\mu_{\Omega,w}$ denote any probability measure on $\Real^n$ of the form:
\[
\mu_{\Omega,w} := \Psi(\norm{x}_\Omega) 1_{\abs{\norm{x}_{\Omega} - 1} \leq w} dx ,
\]
for some Borel function $\Psi : \Real_+ \rightarrow \Real_+$ and $w > 0$. Then for any norm $\norm{\cdot}_0$ on $\Real^n$:
\begin{equation} \label{eq:W1-conclusion}
W_{\norm{\cdot}_0,1}(\mu_{\Omega,w},\sigma_{\partial \Omega}) \leq \frac{n+1}{n} w \int \norm{x}_0 d\lambda_\Omega(x) .
\end{equation}
\end{lem}
\begin{proof}
Let $T: \Real^n \setminus \set{0} \rightarrow \partial \Omega$ be defined as $T(x) := \frac{x}{\norm{x}_\Omega}$. Since the density of $\mu_{\Omega,w}$ depends only on $\norm{x}_{\Omega}$, it is clear that $T$ pushes forward $\mu_{\Omega,w}$ onto the cone measure $\sigma_{\partial \Omega}$. Now consider the probability measure $\pi$ on $\Real^n \times \Real^n$ defined by pushing forward $\mu_{\Omega,w}$ via $\text{Id} \times T$, having first and second marginals precisely $\mu_{\Omega,w}$ and $\sigma_{\partial \Omega}$, respectively. It follows by definition that:
\begin{align*}
& W_{\norm{\cdot}_0,1}(\mu_{\Omega,w},\sigma_{\partial \Omega})  \leq \int \norm{y - x}_{0} d\pi(x,y) = \int \norm{T(x) - x}_{0} d\mu_{\Omega,w}(x) \\
& = \int \abs{\norm{x}_\Omega - 1} \norm{T(x)}_{0} d\mu_{\Omega,w}(x) \leq w \int \norm{T(x)}_{0} d\mu_{\Omega,w}(x) = w \int \norm{y}_0 d\sigma_{\partial \Omega}(y) \\
& = w \int \frac{\norm{z}_0}{\norm{z}_\Omega} d\lambda_{\Omega}(z) = w \frac{n+1}{n} \int \norm{z}_0 d\lambda_{\Omega}(z) ,
\end{align*}
where the last equality may be easily verified e.g. by integration in polar coordinates. 
\end{proof}

\begin{rem}
In fact, when $\Psi : \Real_+ \rightarrow \Real_+$ is a log-concave function so that:
\[
\int_{0 \vee (1-w)}^{1+w} \Psi(t) t^n dt = \int_{0 \vee (1-w)}^{1+w} \Psi(t) t^{n-1} dt ,
\]
(and in particular for the function $\Psi(t) = \frac{1}{Z}\exp(-n t)$ when $w = \infty$), one can do better than just using the very crude estimate $\abs{\norm{x}_\Omega - 1} \leq w$ as we did above. In that case, it is not very hard to show that one may replace $w$ by $\min(w , \frac{C}{\sqrt{n}})$ in (\ref{eq:W1-conclusion}), for an appropriate universal constant $C>0$. 
Since in this work we will only be interested in the range $w \leq \frac{1}{\sqrt{n}}$, we have chosen to only provide the most elementary estimate (\ref{eq:W1-conclusion}).
\end{rem}

\subsection{From Cone to Uniform Measure}

The $L^2$ version of the following Hardy-type inequality was proved by the authors in \cite[Theorem 1]{KolesnikovEMilman-HardyKLS}, reducing various spectral-gap questions from $\Omega$ to its boundary. We will require the following $L^1$ version, which in fact is more elementary. For completeness, we formulate it with respect to an arbitrary norm. 
\begin{lem} \label{lem:Hardy}
Let $\Omega \subset \Real^n$ denote a star-shaped body. Then for any Lipschitz function $f : \Omega \rightarrow \Real$ and any norm $\norm{\cdot}_0$ on $\Real^n$ we have:
\begin{equation} \label{eq:Hardy-gen}
\int \abs{f - \med_{\lambda_\Omega} f} d\lambda_{\Omega} \leq \frac{1}{n} \int \norm{x}_0 \norm{\nabla f}_0^* d\lambda_{\Omega} + \int_{\partial \Omega} \abs{f - \med_{\sigma_{\partial \Omega}} f} d\sigma_{\partial \Omega}  .
\end{equation}
In particular, for any $1$-Lipschitz function $f : (\Real^n,\norm{\cdot}_0) \rightarrow \Real$:
\begin{equation} \label{eq:Hardy-for-Lip}
\int \abs{f - \med_{\lambda_\Omega} f} d\lambda_{\Omega} \leq \frac{1}{n} \int \norm{x}_0 d\lambda_{\Omega} + \int_{\partial \Omega} \abs{f - \med_{\sigma_{\partial \Omega}} f} d\sigma_{\partial \Omega}  .
\end{equation}
\end{lem}
\begin{proof}
Integrating by parts (see e.g. \cite[12.2]{MorganBook4Ed}), we have for any smooth (and in fact, Lipschitz) vector field $\xi$ and function $g$ on $\Omega$:
\[
\int_{\Omega} div(\xi) g dx = - \int_{\Omega} \scalar{\xi , \nabla g} dx + \int_{\partial \Omega} \scalar{\xi,\nu} g \; d\H^{n-1} .
\]
Applying this to $\xi(x) = x$, we obtain:
\[
n \int_{\Omega} g dx \leq \int_{\Omega} \norm{x}_0 \norm{\nabla g}_0^* dx  + \int_{\partial \Omega} \scalar{x,\nu} g \; d\H^{n-1} .
\]
Setting $g = \abs{f - \med_{\sigma_{\partial \Omega}} f}$ and using that $\norm{\nabla g}_0^* \leq \norm{\nabla f}_0^*$, it follows that:
\[
\int_{\Omega}\abs{f - \med_{\sigma_{\partial \Omega}} f} d\lambda_\Omega \leq \frac{1}{n} \int_{\Omega} \norm{x}_0 \norm{\nabla f}^*_0 d\lambda_\Omega + \int_{\partial \Omega} \abs{f - \med_{\sigma_{\partial \Omega}} f} d\sigma_{\partial \Omega} .
\]
Finally, the left-hand-side cannot increase if we replace $\med_{\sigma_{\partial \Omega}} f$ by $\med_{\lambda_\Omega} f$ there, yielding the assertion.  
\end{proof}

\begin{rem}
It is also possible to obtain the particular case (\ref{eq:Hardy-for-Lip}) by estimating $W_1(\nu_{\partial \Omega},\lambda_{\Omega})$ as in the previous subsection, but our proof above has the advantage that it yields the more general (\ref{eq:Hardy-gen}). \end{rem}

\section{Putting Everything Together} \label{sec:proofs}

\subsection{Proof of Main Technical Theorem} \label{subsec:together}

We are now ready to present the proof of our Main Technical Theorem \ref{thm:main} by putting all of the ingredients from the previous sections together. 

\smallskip

We first recall the following well-known facts about one-dimensional log-concave measures. Note that for any probability measure $\nu$ on $(\Real^n,\abs{\cdot})$:
\begin{equation} \label{eq:CovOp}
D_{\Poin}^{\Lin}(\nu)^2 = \norm{\text{Cov}_\nu}_{op} ,
\end{equation}
where $\snorm{\text{Cov}_\nu}_{op}$ denotes the operator norm of $\text{Cov}_\nu$ regarded as a linear operator. Consequently, in the one dimensional case we have $D_{\Poin}^{\Lin}(\nu)^2 = \Var(X)$ where $X$ is distributed according to $\nu$. 
\begin{lem} \label{lem:1D}
Let $\nu = f(x) dx$ denote a log-concave probability measure on $\Real$. Then:
\begin{enumerate}
\item The KLS conjecture is valid: $1 \leq D_{\Poin}(\nu)^2 / D_{\Poin}^{\Lin}(\nu)^2  \leq 12$. 
\item We have $C_1 \leq \norm{f}^2_{L^\infty} D_{\Poin}^{\Lin}(\nu)^2 \leq C_2$ for two universal constants $C_1,C_2 > 0$. 
\end{enumerate}
\end{lem}
\begin{proof}
The first assertion is due to Bobkov \cite[Corollary 4.3]{BobkovGaussianIsoLogSobEquivalent}.
The second one may be found in \cite{Milman-Pajor-LK} when $f$ is even and in \cite[Theorem 4]{FradeliziHyperplaneSections} or \cite[Lemmas 2.5 and 2.6]{KlartagPerturbationsWithBoundedLK} 
in the general case.
\end{proof}

Recalling the assumptions of Theorem \ref{thm:main}, we are given that for each $i=1,\ldots,n$, $\mu_i := \exp(-V_i(y)) dy$ is a log-concave probability measure on $\Real$ with $\min V_i = m_i$. It follows by Lemma \ref{lem:1D} that $D_{\Poin}(\mu_i) \leq \sqrt{12 C_2} e^{m_i}$ for every $i$. Since $\mu := \mu_1 \otimes \ldots \otimes \mu_n$ is a product measure, by the well-known tensorization property of the Poincar\'e inequality (e.g. \cite{Ledoux-Book}), we conclude that $D_{\Poin}(\mu) \leq \sqrt{12 C_2} M$, where recall $M = \max_{i=1,\ldots,n} e^{m_i} \geq 1$. 

Next, given a probability measure $\nu$ on (say) $\Real^n$, denote $\K_\nu = \K(\Real^n,\abs{\cdot},\nu)$. By a well-known result of M. Gromov and V. Milman \cite{GromovMilmanLevyFamilies} (see also \cite[Corollary 2.7]{EMilman-RoleOfConvexity}), a Poincar\'e inequality always implies the following exponential concentration:
\begin{equation} \label{eq:GM0}
\K_\nu(r) \leq \exp\brac{ - c_0 r / D_{\Poin}(\nu) } \;\;\; \forall r > 0 ~,
\end{equation}
for some universal numeric constant $c_0 > 0$. In fact, it is possible to use any $c_0 \in (0,2)$ at the expense of using an additional multiplicative constant in front of the right-hand-side above (see \cite{Schmucky-PoincareAndMartingales,BobkovLedouxModLogSobAndPoincare}), but we will not require this here. 
It follows that for our measure $\mu$, we have for some numeric constant $c > 0$:
\begin{equation}\label{eq:GM}
\K_\mu(r) \leq \exp(- c r / M) \;\;\; \forall r > 0 .
\end{equation}

\smallskip

Applying Proposition $\ref{prop:Lp}$ with, say $p=2$, we obtain the following estimate, valid for all $w_0 \in (0,1/2]$:
\begin{equation} \label{eq:L2}
\norm{\frac{d\mu_{K_{E,\frac{w_0}{\sqrt{n}}}}}{d\mu}}_{L^2(\mu)} \leq \frac{C'}{Z_E} \frac{1}{w_0} \exp(16 w_0^2 (A^{(\infty,2)})^2) \;\;\; \forall E > E_0 .
\end{equation}
By Proposition \ref{prop:ZE}, we know that $Z_E \geq \frac{1}{e}$ on the entire non-empty closed interval $E \in [E_{\min},E_{\max}]$. 
Invoking Proposition \ref{prop:BartheMilman}, the resulting estimate (\ref{eq:L2}) allows us to transfer the concentration estimate (\ref{eq:GM}) from $\mu$ onto its linearized version $\mu_{K_E,\frac{w_0}{\sqrt{n}}}$ on the corresponding annulus, yielding for all $E \in [E_{\min},E_{\max}] \cap (E_0,\infty)$:
\[
\K_{\mu_{K_E,\frac{w_0}{\sqrt{n}}}}(r) \leq \frac{2 e C'}{w_0} \exp(16 w_0^2 (A^{(\infty,2)})^2) \exp\brac{- \frac{c r}{4M}} \;\;\; \forall r > 0 .
\]
In particular, for any $1$-Lipschitz function $f$ on $(\Real^n,\abs{\cdot})$ and any $r_0 \geq 0$, we have:
\begin{align*}
\int \abs{f - \med_{\mu_{K_E,\frac{w_0}{\sqrt{n}}}} f} d\mu_{K_E,\frac{w_0}{\sqrt{n}}} & \leq r_0 + 2 \int_{r_0}^{\infty} \K_{\mu_{K_E,\frac{w_0}{\sqrt{n}}}}(r) dr \\
& \leq r_0 + \frac{16 e C' M}{c w_0} \exp\brac{16 w_0^2 (A^{(\infty,2)})^2 - \frac{c r_0}{4M}} .
\end{align*}
Optimizing on $r_0$ (after recalling that $M \geq 1$ and $w_0 \leq 1/2$), we deduce for an appropriate numeric constant $C''>0$:
\begin{equation} \label{eq:FM1}
\int \abs{f - \med_{\mu_{K_E,\frac{w_0}{\sqrt{n}}}} f} d\mu_{K_E,\frac{w_0}{\sqrt{n}}} \leq C'' M \brac{w_0^2 (A^{(\infty,2)})^2  + \log \frac{M}{w_0}} .
\end{equation}

\smallskip
Next, by Lemma \ref{lem:W1-estimate} applied to $K_E$ and $\mu_{K_E,w}$ (with $\Psi(t) = 1_{t \in [1-w,1]}$), we know that for all $w > 0$:
\[
W_{\abs{\cdot},1}(\mu_{K_E,w},\sigma_{\partial K_E}) \leq \frac{n+1}{n} w \int \abs{x} d\lambda_{K_E} . 
\]
Invoking Lemma \ref{lem:W1}, the latter estimate allows us to transfer the first-moment concentration (\ref{eq:FM1}) from $\mu_{K_E,\frac{w_0}{\sqrt{n}}}$ onto the cone measure $\sigma_{\partial K_E}$, yielding for any $1$-Lipschitz function $f$ on $(\Real^n,\abs{\cdot})$:
\[
\int \abs{f - \med_{\sigma_{\partial K_E}} f} d\sigma_{\partial K_E} \leq C'' M \brac{w_0^2 (A^{(\infty,2)})^2 + \log \frac{M}{w_0}}  + \frac{n+1}{n} \frac{w_0}{\sqrt{n}} \int \abs{x} d\lambda_{K_E} .
\]

\smallskip
Finally, we invoke Lemma \ref{lem:Hardy} to transfer the latter first-moment concentration from $\sigma_{\partial K_E}$ to $\lambda_{K_E}$, yielding for all $w_0 \in (0,1/2]$: 
\[
\int \abs{f - \med_{\lambda_{K_E}} f} d\lambda_{K_E} \leq C'' M \brac{w_0^2 (A^{(\infty,2)})^2 + \log \frac{M}{w_0}}  + \brac{\frac{n+1}{n} \frac{w_0}{\sqrt{n}} + \frac{1}{n}} \int \abs{x} d\lambda_{K_E} .
\]
Optimizing on $w_0$, we set $w_0 := \frac{1}{2 A^{(\infty,2)}} \in (0,1/2]$ (recall that by definition $A^{(\infty,2)} \geq 1$), obtaining:
\[
\int \abs{f - \med_{\lambda_{K_E}} f} d\lambda_{K_E} \leq C''' M \log (e + A^{(\infty,2)} M)  + \frac{2}{\sqrt{n}} \int \abs{x} d\lambda_{K_E} .
\]

\smallskip
It remains to invoke the following result, established in \cite{EMilman-RoleOfConvexity} in a more general weighted Riemannian setting (see also \cite{EMilmanGeometricApproachPartI,EMilmanIsoperimetricBoundsOnManifolds,EMilmanNegativeDimension} for refinements), asserting the equivalence between concentration, spectral-gap and linear-isoperimetry under appropriate convexity assumptions:
\begin{thm}[\cite{EMilman-RoleOfConvexity}] \label{thm:IsopConcEquiv}
For any log-concave probability measure $\nu$ on $\Real^n$:
\[
D_{\Poin}(\nu) \leq C  \sup \set{ \int \abs{f - \med_\nu f} d\nu \; ; \; \text{$f  : (\Real^n,\abs{\cdot}) \rightarrow \Real$ is $1$-Lipschitz }  } ,
\]
with some universal numeric constant $C > 1$. 
\end{thm}
\noindent
As $K_E$ is convex and hence $\lambda_{K_E}$ is a log-concave measure, this verifies the first assertion of Theorem \ref{thm:main}:
\[
 D_{\Poin}(\lambda_{K_E}) \leq C \brac{M \log(e+ A^{(\infty,2)} M) + \frac{1}{\sqrt{n}} \int \abs{x} d\lambda_{K_E}} .
\] 
The second assertion follows since by the triangle and Jensen inequalities:
\[
\frac{1}{\sqrt{n}} \int \abs{x} d\lambda_{K_E} \leq \frac{\abs{b_E}}{\sqrt{n}} + \frac{1}{\sqrt{n}} \int \abs{x-b_E} d\lambda_{K_E} \leq \frac{\abs{b_E}}{\sqrt{n}} + \brac{\frac{1}{n} \int \abs{x-b_E}^2 d\lambda_{K_E}}^{1/2} ,
\]
and:
\begin{equation} \label{eq:PoinLinVsNorm}
\frac{1}{n} \int \abs{x-b_E}^2 d\lambda_{K_E} = \frac{1}{n} \text{tr} \; \text{Cov}_{\lambda_{K_E}} \leq \snorm{\text{Cov}_{\lambda_{K_E}}}_{op} = D_{\Poin}^{\Lin}(\lambda_{K_E})^2 . 
\end{equation}

\subsection{Proof of Theorem \ref{thm:logn}} \label{subsec:together2}

The proof of Theorem \ref{thm:logn} is identical to the one of Theorem \ref{thm:main} described in the previous subsection, with the only difference being in the first step -- instead of invoking the $L^p$ estimate given by Proposition \ref{prop:Lp} for transferring concentration from $\mu$ to $\mu_{K_E,w}$, we invoke the $L^\infty$ estimate of Lemma \ref{lem:Linfty}. Let us sketch the argument.

\smallskip
By translating $\mu$ we may assume that $V(0) = \min V$, where recall the latter value is assumed to be $0$. 
By Lemma \ref{lem:Linfty}, we have for all $w_0 \in (0,1]$:
\begin{equation} \label{eq:Linfty}
\norm{\frac{d\mu_{K_{E,\frac{w_0}{\sqrt{n}}}}}{d\mu}}_{L^\infty(\mu)} \leq \frac{1}{c w_0 Z_E} \exp(w_0 \sqrt{n}) \;\;\; \forall E > 0.
\end{equation}
By Proposition \ref{prop:ZE}, we know that $Z_E \geq \frac{1}{e}$ on the interval $E \in [E_{\min},E_{\max}]$. 
Invoking Proposition \ref{prop:BartheMilman} with $p=\infty$, we transfer the Gromov--Milman concentration (\ref{eq:GM0}) from $\mu$ onto $\mu_{K_E,\frac{w_0}{\sqrt{n}}}$, yielding for all $E \in [E_{\min},E_{\max}]$:
\[
\K_{\mu_{K_E,\frac{w_0}{\sqrt{n}}}}(r) \leq \frac{2 e}{c w_0} \exp(w_0 \sqrt{n}) \exp\brac{- \frac{c_0 r}{2 D_{\Poin}(\mu)}} \;\;\; \forall r > 0 .
\]
The rest of the proof is identical to the one in the previous subsection, with $M$ replaced by $D_{\Poin}(\mu)$ and $w_0^2 (A^{(\infty,2)})^2$ replaced by $w_0 \sqrt{n}$. 
Note that just as with lower bound $M \geq 1$ in the previous subsection, our normalization ensures that $D_{\Poin}(\mu) \geq c > 0$. Indeed:
\[
D_{\Poin}(\mu) \geq D_{\Poin}^{\Lin}(\mu) = \norm{\text{Cov}_\mu}^{1/2}_{op} \geq \brac{\text{det} \; \text{Cov}_\mu}^{1/2n} = L_\mu \geq c > 0 ,
\]
where $L_{\mu}$ denotes the isotropic constant of $\mu$, the last equality holds since we assume that $\mu = \exp(-V) dx$ with $\min V = 0$, and the inequality $L_\mu \geq c > 0$ for all log-concave measures $\mu$ is well-known (see \cite{Milman-Pajor-LK,Klartag-Psi2,GreekBook} for more background on the isotropic constant). 

Repeating the argument in the previous subsection, we obtain for any $1$-Lipschitz function $f$ on $(\Real^n,\abs{\cdot})$:
\[ \int \abs{f - \med_{\mu_{K_E,\frac{w_0}{\sqrt{n}}}} f} d\mu_{K_E,\frac{w_0}{\sqrt{n}}} \leq C'' D_{\Poin}(\mu) \brac{w_0 \sqrt{n} + \log\brac{e +\frac{D_{\Poin}(\mu)}{w_0}}} .
\] Transferring concentration to $\sigma_{\partial K_E}$ and then to $\lambda_{K_E}$ as before, we obtain for all $w_0 \in (0,1]$:
\begin{align*}
\int \abs{f - \med_{\lambda_{K_E}} f} d\lambda_{K_E} & \leq C'' D_{\Poin}(\mu) \brac{w_0 \sqrt{n} + \log\brac{e +\frac{D_{\Poin}(\mu)}{w_0}}}  \\
& + \brac{\frac{n+1}{n} \frac{w_0}{\sqrt{n}} + \frac{1}{n}} \int \abs{x} d\lambda_{K_E} .
\end{align*}
Setting $w_0 = \frac{1}{\sqrt{n}}$, we deduce:
\[
\int \abs{f - \med_{\lambda_{K_E}} f} d\lambda_{K_E} \leq C''' D_{\Poin}(\mu) \log (e + \sqrt{n} D_{\Poin}(\mu)) + \frac{3}{n} \int \abs{x} d\lambda_{K_E}.
\]
Invoking Theorem \ref{thm:IsopConcEquiv}, the assertion of Theorem \ref{thm:logn} will follow as soon as we show that:
\begin{equation} \label{eq:logn-goal}
\frac{1}{n} \int \abs{x} d\lambda_{K_E} \leq C \log(1+n) D_{\Poin}(\mu) 
\end{equation}
(since $D_{\Poin}(\mu) \geq c > 0$). Note that the barycenter of $\lambda_{K_E}$ may not be at the origin. 

\smallskip
To establish (\ref{eq:logn-goal}), note that by \cite[Theorem 4.1]{KLS}, any convex body $K$ in $(\Real^n,\abs{\cdot})$ satisfies:
\[
 K - \int x \; d\lambda_K\subset (n+1) \text{Cov}^{1/2}_{\lambda_K}(B_2^n) ,
\]
where $\text{Cov}^{1/2}_{\lambda_K}$ is considered as a linear map acting on the Euclidean unit-ball $B_2^n$. Since $0 \in K_E$, it follows by (\ref{eq:CovOp}) that:
\[
\frac{1}{n} \int \abs{x} d\lambda_{K_E} \leq \frac{1}{n} \text{diam}(K_E) \leq \frac{n+1}{n} \snorm{\text{Cov}_{\lambda_{K_E}}}_{op}^{1/2} \text{diam}(B_2^n) = 2 \frac{n+1}{n} \snorm{\text{Cov}_{\lambda_{K_E}}}_{op}^{1/2} . \]
But by Proposition \ref{prop:Level} (7):
\[
\snorm{\text{Cov}_{\lambda_{K_E}}}^{1/2}_{op} \leq C \log(1+n) \snorm{\text{Cov}_{\mu}}^{1/2}_{op} = C \log(1+n) D_{\Poin}^{\Lin}(\mu) \leq C \log(1+n) D_{\Poin}(\mu) ,
\]
thereby confirming (\ref{eq:logn-goal}), and hence concluding the proof of Theorem \ref{thm:logn}.

\subsection{Proofs of Remaining Statements}

Let us now conclude the proofs of assertion (8) of Proposition \ref{prop:Level}, Theorem \ref{thm:main2}, and Theorems \ref{thm:intro} and \ref{thm:intro-logn}. 

\begin{proof}[Proof of assertion (8) of Proposition \ref{prop:Level}]
Recalling (\ref{eq:PoinLinVsNorm}) and invoking the well known bath-tub principle (see e.g. \cite{Milman-Pajor-LK}):
\[
D_{\Poin}^{\Lin}(\lambda_{K_E})^2 \geq \frac{1}{n} \int \abs{x-b_E}^2 d\lambda_{K_E}(x) \geq \frac{1}{n} \int \abs{x}^2 d\lambda_{B_E} ,
\]
where $B_E$ is a Euclidean ball centered at the origin and having the same volume as $K_E$. Since $\Vol(B_E)^{1/n} = Vol(K_E)^{1/n} \geq c > 0$ for all $E \geq E_{\min}$ by Proposition \ref{prop:Level} (3), an elementary and well-known computation (see again \cite{Milman-Pajor-LK}) ensures that $D_{\Poin}^{\Lin}(\lambda_{K_E}) \geq c  > 0$ for all $E$ in that range, establishing assertion (8) of that proposition. 
\end{proof}

\begin{proof}[Proof of Theorem \ref{thm:main2}]
By Corollary \ref{cor:main}, we know that for any $E \in \text{Level}(V) = [E_{\min},E_{\max}]$ and  $b \in \text{int}(K_E)$:
\[
 D_{\Poin}(\lambda_{K_E}) \leq  C_1 \brac{M \log(e+ A^{(\infty,2)}(b) M) + \frac{\abs{b - b_{E}}}{\sqrt{n}} + D_{\Poin}^{\Lin}(\lambda_{K_E})} .
 \]
The additive dependence in $D_{\Poin}^{\Lin}(\lambda_{K_E})$ above turns into a multiplicative one by changing the numerical constant $C_1$ and using that $D_{\Poin}^{\Lin}(\lambda_{K_E}) \geq c  > 0$ for all $E \geq E_{\min}$ and that $M \geq 1$. Whenever $b = b_\mu$ lies in $\text{int}(K_E)$ for $E \in \text{Level}(V)$,
we have by Proposition \ref{prop:Level} (6) that $\frac{\abs{b_\mu - b_{E}}}{\sqrt{n}} = o(1)$ as $n \rightarrow \infty$, and so this term may be discarded at the expense of changing again numerical constants. This is indeed the case whenever $(E_V-1) \vee E_{\min} < E \leq E_{\max}$ by Proposition \ref{prop:Level} (\ref{assertion:Fradelizi}), and we may also take $E = (E_V-1) \vee E_{\min}$ by a limiting argument. Note that necessarily $E_V \in [(E_V-1) \vee E_{\min} , E_{\max}]$ by Proposition \ref{prop:Level} (4). 
Finally, Lemma \ref{lem:Fradelizi2} implies that $A^{(\infty,2)}(b_\mu) \leq \sqrt{2} A^{(2)}(b_\mu)$, and so the assertion follows by a final adjustment of constants. 
\end{proof}

\begin{proof}[Proof of Theorems \ref{thm:intro} and \ref{thm:intro-logn}]
The dimension-independent part of the estimate of Theorem \ref{thm:intro} immediately follows from an application of Theorem \ref{thm:main2} for any $E \in [E_{\min},E_{\max}]$ (since by assumption $b_\mu = 0 \in \text{int}(K_E)$). 
The dimension-dependent part follows by Theorem \ref{thm:logn} applied to $\mu = \mu_1 \otimes \ldots \otimes \mu_n$ since $D_{\Poin}(\mu) \leq C$ as explained in Subsection \ref{subsec:together} and since $D_{\Poin}^{Lin}(\lambda_{K_E}) \geq c$ for all $E \geq E_{\min}$ by Proposition \ref{prop:Level} (8). The volume estimate (\ref{eq:intro-vol}) follows by Proposition \ref{prop:Level} (3). Similarly, Theorem \ref{thm:intro-logn} holds for all $E \in [E_{\min},E_{\max}]$ by Theorem \ref{thm:logn}. 
\end{proof}

\subsection{General Formulation After Rescaling} \label{subsec:gen}

\begin{cor}[Main Theorem - Generalized Version] \label{cor:main-scale}
Let $\tilde{W}_i : \Real \rightarrow \Real$, $i=1,\ldots,n$, denote a sequence of convex functions normalized so that $\min \tilde{W}_i = 0$. Assume that $z_i := \int \exp(-\tilde{W}_i(y)) dy < \infty$ and set $\tilde{V_i} = \tilde{W_i} + \log z_i$, $\tilde{\mu}_i = \exp(-\tilde{V_i}(y)) dy$, $\tilde{V}(x) = \sum_{i=1}^n \tilde{V_i}(x_i)$ and $\tilde{\mu} = \exp(-\tilde{V}(x)) dx$.  

Given $\tilde{b} \in \Real^n$, let $\alpha^{(2)}_i = \alpha^{(2)}_i(\tilde{V},\tilde{b})$ and $\alpha^{(\infty,2)}_i = \alpha^{(\infty,2)}_i(\tilde{V},\tilde{b})$ be given by:
\[
\alpha^{(2)}_i := \snorm{\tilde{V}_i'(y) (y-\tilde{b}_i)}_{L^2(\tilde{\mu}_i)} ~,~ \alpha^{(\infty,2)}_i := (1 + \snorm{(\tilde{V}_i'(y) (y-\tilde{b}_i))_{-}}_{L^\infty(\tilde{\mu}_i)}) \vee \alpha^{(2)}_i ,
\]
and set:
\[
A^{(2)}(\tilde{b}) := \frac{1}{\sqrt{n}}\snorm{(\alpha^{(2)}_i(\tilde{V},\tilde{b}))_{i=1}^n}_2 ~,~ A^{(\infty,2)}(\tilde{b})  := \frac{1}{\sqrt{n}}\snorm{(\alpha^{(\infty,2)}_i(\tilde{V},\tilde{b}))_{i=1}^n}_2 .
\]
Denote $z := (\Pi_{i=1}^n z_i)^{1/n}$, $M = \max_{i=1,\ldots,n} \frac{z_i}{z}$.
Set $\tilde{W}(x) := \sum_{i=1}^n \tilde{W}_i(x_i)$ and $V(x) = \tilde{W}(z x)$. 
Given $E > 0$, consider the convex sets:
\[
\tilde{K}_E := \set{W \leq E } = \set{ x \in \Real^n \; ; \; \sum_{i=1}^n \tilde{W}_i(x_i) \leq E} ~,~ K_E := \set{V \leq E} = \frac{1}{z} \tilde{K}_E .
\]
Denote as usual:
\[
\text{Level}(V) := \set{E \geq 0 \; ; \; e^{-E} \Vol(K_E) \geq \frac{1}{e} \frac{n^n e^{-n}}{n!}} .
\]
Then all of the assertions of Proposition \ref{prop:Level} apply to $\text{Level}(V)$, and we have for all $E \in \text{Level}(V)$ and $\tilde{b} \in int(\tilde{K}_E)$:
 \[
 D_{\Poin}(\lambda_{\tilde{K}_E}) \leq C \brac{z M \log(e+ A^{(\infty,2)}(\tilde{b}) M) + \frac{|\tilde{b} - \tilde{b}_{E}|}{\sqrt{n}} + D_{\Poin}^{\Lin}(\lambda_{\tilde{K}_E})} ,
\]
where $\tilde{b}_E := \int x \; d\lambda_{\tilde{K}_E}$.
In addition, setting $\tilde{b} = \tilde{b}_{\tilde{\mu}} := \int x d\tilde{\mu}$, we have:
\[
D_{\Poin}(\lambda_{\tilde{K}_E}) \leq C' M \log(e+ A^{(2)}(\tilde{b}_{\tilde{\mu}}) M) D_{\Poin}^{\Lin}(\lambda_{\tilde{K}_E}) \;\;\; \forall E \in [(E_V - 1) \vee E_{\min} , E_{\max}] . 
\]
\end{cor}
\begin{proof}Denote $V_i(y) := \tilde{V}_i(z y) - \log z$, and note that both $\exp(-\tilde{V}_i) dy$ and $\exp(-V_i) dy$ are probability measures on $\Real$.
Also note that: 
\[
V(x) = \tilde{W}(z x) = \sum_{i=1}^n \tilde{W}_i(z x_i) =  \sum_{i=1}^n V_i(x_i)  ,
\]
and since $\tilde{V}(x) = \sum_{i=1}^n \tilde{V}_i(x_i)$, we see that the probability measure $\tilde{\mu} = \exp(-\tilde{V}) dx$ on $\Real^n$ is obtained by scaling $\mu := \exp(-V) dx$ by a factor of $z$. Lastly, note that:
\[
M = \max_{i=1,\ldots,n} \frac{z_i}{z} = \max_{i=1,\ldots,n} e^{\min V_i},
\]
and that:
\[
 \min V = \sum_{i=1}^n \min V_i = \sum_{i=1}^n (\log z_i - \log z) = 0.
\]

Consequently, we may apply Corollary \ref{cor:main} and Theorem \ref{thm:main2} to the measure $\mu$, the functions $\set{V_i}$ and the associated levels sets $K_E$. 
By scale invariance, we have that $\alpha_i(V , b) = \alpha_i(\tilde{V} , \tilde{b})$ for all $b \in \Real^n$ and  $\tilde{b} = z b$. Clearly $\tilde{b}_{\tilde{\mu}} = z b_\mu$ where $\tilde{b}_{\tilde{\mu}} = \int x d\tilde{\mu}$ and $b_\mu = \int x d\mu$. The assertions for $\tilde{K}_E$ now immediately follow after taking into account that $\tilde{K}_E = z K_E$, implying that $D_{\Poin}(\lambda_{\tilde{K}_E}) = z D_{\Poin}(\lambda_{K_E})$ and $D_{\Poin}^{\Lin}(\lambda_{\tilde{K}_E}) = z D_{\Poin}^{\Lin}(\lambda_{K_E})$.
\end{proof}

\subsection{Confirmation of Examples \ref{ex:1} and \ref{ex:2}} \label{subsec:examples}

The assertion of Example \ref{ex:1} for $\tilde{W}_i(x_i) = (x_i)_+^{p^+_i} + (x_i)_-^{p^-_i}$ with $p_i^{\pm} \in [1,P]$ follows from Corollary \ref{cor:main-scale}. Let us prove this in the generality suggested to us by the referee: we assume that the convex functions $\tilde{W}_i$ satisfy $\min \tilde{W}_i(y) = \tilde{W}_i(0) = 0$, that:
\[ \forall i=1,\ldots,n  \;\;\; 0 < c_1 \leq \int_0^\infty \exp(-\tilde{W}_i(\pm x_i)) dx_i \leq c_2  < \infty , 
\] and that the following ``generalized doubling condition" holds:
\[ \forall i=1,\ldots,n \;\; \; \exists \eps_i > 0 \;\; \; \forall x_i \in \Real \;\;\;  \tilde{W}_i((1+\eps_i)x_i) \leq (1+\eps_i P) \tilde{W}_i(x_i)  .
 \]  The latter condition's sole purpose is to ensure (by convexity) that:
\begin{equation} \label{eq:referee-estimate}
\forall x_i \in \Real \;\;\; 0 \leq \tilde{W}_i'(x_i) x_i \leq \frac{\tilde{W}_i((1+\eps_i)x_i) - \tilde{W}_i(x_i)}{\eps_i} \leq P \tilde{W}_i(x_i) ,
\end{equation}
from whence the extremality of the function $\tilde{W}_i(x_i) = \abs{x_i}^P$ is clearly apparent. 

Denote  $z_i := \int \exp(-\tilde{W}_i(y)) dy$ so that $2 c_1 \leq z_i \leq 2 c_2$ for all $i=1,\ldots,n$. Consequently $2 c_1 \leq z := (\Pi_{i=1}^n z_i)^{1/n} \leq 2 c_2$, $M = \max_{i=1,\ldots,n} \frac{z_i}{z} \leq \frac{c_2}{c_1}$, and Lemma \ref{lem:1D} ensures that the probability measures $\tilde{\mu}_i := \frac{1}{z_i} \exp(-\tilde{W}_i(y)) dy$ satisfy $D_{\Poin}(\tilde{\mu}_i) \leq C$ (independently of $P$).

Note that $y \mapsto \tilde{W}_i(\pm y)$ and $y \mapsto |\tilde{W}_i'(\pm y)|$ are non-decreasing functions on $[0,\infty)$ by unimodality and convexity, respectively. 
Denoting the barycenter $\tilde{b}_i := \int y d\tilde{\mu}_i(y)$, recall that by Fradelizi's estimate (Remark \ref{rem:Fradelizi}) $\tilde{W}_i(\tilde{b}_i) \leq \min \tilde{W}_i + 1 = 1$, and so by unimodality $|\tilde{b}_i| \leq c_2 e$. In addition, convexity implies that $| \tilde{W}_i'(\pm \frac{c_1}{2})| \leq \frac{2}{c_1}$, since otherwise we would have $\int_0^\infty \exp(-\tilde{W}_i(\pm x_i)) dx_i < c_1$. 

We now arrive to the main calculation. Invoking (\ref{eq:referee-estimate}):
\[
\norm{\tilde{W}_i'(x_i) (x_i - \tilde{b}_i) }_{L^{2}(\tilde{\mu}_i)} \leq |\tilde{b}_i | \snorm{\tilde{W}_i'}_{L^2(\tilde{\mu}_i)} + P \snorm{\tilde{W}_i}_{L^2(\tilde{\mu}_i)} . 
\]
Now:
\[
\int |\tilde{W}_i'(x_i)|^2 d\tilde{\mu}_i(x_i) \leq \int \frac{4}{c^2_1} (1 +|\tilde{W}_i'(x_i)|^2 x_i^2) d\tilde{\mu}_i(x_i)  ,
\]
and so by (\ref{eq:referee-estimate}) again, we conclude:
\[
\norm{\tilde{W}_i'(x_i) (x_i - \tilde{b}_i) }_{L^{2}(\tilde{\mu}_i)} \leq \frac{2 c_2 e}{c_1} \sqrt{1 + P^2 \snorm{\tilde{W}_i}^2_{L^2(\tilde{\mu}_i)}} + P \snorm{\tilde{W}_i}_{L^2(\tilde{\mu}_i)}  .
\]
Finally, using the inequality $\frac{t^2}{2} \leq e^t$ for $t \geq 0$, and $\tilde{W}_i(x_i)/2 \geq \tilde{W}_i(x_i/2)$, we obtain:
\begin{align*}
&\snorm{\tilde{W}_i}^2_{L^2(\tilde{\mu}_i)} = \frac{1}{z_i} \int \tilde{W}^2_i \exp(-\tilde{W}_i) dx_i \leq \frac{8}{z_i}\int \exp(-\tilde{W}_i/2) dx_i \\
& \leq \frac{8}{z_i}\int \exp(-\tilde{W}_i(x_i/2)) dx_i = \frac{16}{z_i} \int \exp(-\tilde{W}_i(y)) dy \leq \frac{16 c_2}{c_1} . 
\end{align*}
It follows that $A^{(2)}(\tilde{b}_{\tilde{\mu}}) \leq C P$, where $\tilde{b}_{\tilde{\mu}} = (\tilde{b}_1,\ldots,\tilde{b}_n)$ is the barycenter of $\tilde{\mu} = \tilde{\mu}_1 \otimes \ldots \otimes\tilde{\mu}_n$, and $C$ depends solely on $c_1,c_2$. Invoking Corollary \ref{cor:main-scale}, we deduce that for all $E \in [(E_V-1) \vee E_{\min}, E_{\max}]$ we have:
\[
D_{\Poin}(\lambda_{\tilde{K}_E}) \leq C'' \log(e + P) D_{\Poin}^{\Lin}(\lambda_{\tilde{K}_E}) , 
\]
for:
\[
\tilde{K}_E := \set{ x \in \Real^n \; ; \; \sum_{i=1}^n \tilde{W}_i(x_i) \leq E} . 
\]
Here $E_V,E_{\min},E_{\max}$ refer to $V(x) = \tilde{W}(z x)$ where $\tilde{W}(x) = \sum_{i=1}^n \tilde{W}_i(x_i)$. Note that:
\begin{align*}
E_V - 1 &= \int_{\Real^n} V(x) e^{-V(x)} dx = \frac{1}{z^n} \int_{\Real^n} \tilde{W}(x) e^{-\tilde{W}(x)} dx = \int_{\Real^n} \sum_{i=1}^n \tilde{W}_i(x) e^{-\sum_{i=1}^n \tilde{V}_i(x)} dx \\
& = \sum_{i=1}^n \int_{\Real} \tilde{W}_i(y) e^{-\tilde{V}_i(y)} dy = \sum_{i=1}^n \E \tilde{W}_i(X_i) ,
\end{align*}
where $X_i$ are distributed according to $\tilde{\mu_i} = \exp(-\tilde{V}_i(y)) dy = \frac{1}{z_i} \exp(-\tilde{W}_i(y)) dy$. 
Finally, since $\tilde{K}_E = z K_E$ and $2c_1 \leq z \leq 2 c_2$, the volume estimate (\ref{eq:intro-vol}) for $\tilde{K}_E$ follows from the one ensured for $K_E$ by Proposition \ref{prop:Level} (3).

\smallskip

When $p_i^{\pm} \in [2,P]$, the above estimates may in fact be improved -- we briefly sketch the argument. In this range, the measures $\tilde{\mu}_i$ in fact satisfy a log-Sobolev inequality independently of $P$ (for instance, since they are Lipschitz images of the Gaussian measure - see e.g. \cite{LatalaJacobInfConvolution}). By the tensorization property of the log-Sobolev inequality, it follows that the measures $\tilde{\mu}$ and $\mu$ also satisfy the log-Sobolev inequality with a universal constant independent of $P$ or $n$, and so by the Herbst argument satisfy a Gaussian-type concentration, instead of just an exponential one:
\[
\K_{\mu}(r) \leq \exp(- c r^2) \;\;\; \forall r > 0 ;
\]
we refer to \cite{Ledoux-Book} for more on the log-Sobolev inequality and the Herbst argument. 
Repeating the analysis in Subsections \ref{subsec:together} and \ref{subsec:together2}, one may check that results in a square-root improvement of the previous logarithmic estimates. 

\smallskip

Lastly, Example \ref{ex:2} is an immediate consequence of Theorem \ref{thm:intro-logn}, since when $\text{Hess} V \geq \rho^2 \text{Id}$ for $\rho > 0$, the Bakry--\'Emery criterion \cite{BakryEmery} ensures in particular that $D_{\Poin}(\mu) \leq \frac{1}{\rho}$.

\bibliographystyle{plain}
\bibliography{../../../ConvexBib}

\end{document}